\documentclass[11pt,preprint,numbers,nopreprintline]{elsarticle}

\usepackage[T1]{fontenc}
\usepackage{lmodern}
\usepackage{amsmath,amssymb,amsthm}
\usepackage{mathtools}
\usepackage{enumitem}
\usepackage{aliascnt}
\usepackage[protrusion=true,expansion=false]{microtype}
\usepackage[margin=1in,footskip=24pt]{geometry}
\usepackage[hidelinks]{hyperref}
\usepackage[nameinlink,noabbrev,capitalize]{cleveref}
\setlist[enumerate]{leftmargin=2.2em,itemsep=2pt,topsep=4pt}
\setlist[itemize]{leftmargin=2em,itemsep=2pt,topsep=4pt}
\allowdisplaybreaks
\newtheorem{theorem}{Theorem}[section]
\newaliascnt{proposition}{theorem}
\newtheorem{proposition}[proposition]{Proposition}
\aliascntresetthe{proposition}
\newaliascnt{lemma}{theorem}
\newtheorem{lemma}[lemma]{Lemma}
\aliascntresetthe{lemma}
\newaliascnt{corollary}{theorem}
\newtheorem{corollary}[corollary]{Corollary}
\aliascntresetthe{corollary}
\theoremstyle{definition}
\newaliascnt{definition}{theorem}
\newtheorem{definition}[definition]{Definition}
\aliascntresetthe{definition}
\newaliascnt{convention}{theorem}

\aliascntresetthe{convention}
\theoremstyle{remark}
\newaliascnt{remark}{theorem}
\newtheorem{remark}[remark]{Remark}
\aliascntresetthe{remark}
\newaliascnt{example}{theorem}
\newtheorem{example}[example]{Example}
\aliascntresetthe{example}

\newcommand{\NN}{\mathbb{N}}
\newcommand{\FF}{\mathcal{F}}
\newcommand{\HH}{\mathbb{H}}
\newcommand{\PhiPlus}{\Phi^{+}}
\newcommand{\rank}{\operatorname{rk}}
\newcommand{\height}{\operatorname{ht}}
\newcommand{\bottom}{\widehat{0}}
\newcommand{\wideP}{\widehat{P}}
\newcommand{\wideL}{\widehat{L}}
\newcommand{\qsum}[1]{\sigma(#1)}
\newcommand{\qbinom}[2]{\genfrac{[}{]}{0pt}{}{#1}{#2}_q}

\hypersetup{
  pdftitle={q-Zeta positivity for root and Ferrers-cell posets},
  pdfauthor={Qihang Wang and Weiye Li},
  pdfsubject={Algebraic combinatorics and finite posets},
  pdfkeywords={root poset, q-Zeta numerator, flag h-vector, EL-labelling, ballot word}
}
\begin{document}
\begin{frontmatter}

\title{\(q\)-Zeta positivity for root and Ferrers-cell posets}
\author[pku]{Qihang Wang\corref{corresponding}\fnref{equal}}
\ead{2501110049@stu.pku.edu.cn}
\author[tsinghua]{Weiye Li\fnref{equal}}
\ead{liwy25@mails.tsinghua.edu.cn}
\affiliation[pku]{organization={School of Mathematical Sciences, Peking University},
  city={Beijing}, postcode={100871}, country={China}}
\affiliation[tsinghua]{organization={Department of Automation, Tsinghua University},
  city={Beijing}, postcode={100084}, country={China}}
\cortext[corresponding]{Corresponding author.}
\fntext[equal]{Qihang Wang and Weiye Li contributed equally to this work.}

\begin{abstract}
We prove that Chapoton's \(q\)-Zeta numerator has nonnegative coefficients
for the positive-root poset of every finite reduced irreducible
crystallographic root system, graded by root height minus one and
normalized by Chapoton's prescribed denominator.  The same holds for every
interval of the root poset with zero adjoined.  The key observation is
that adjoining a minimum changes the numerator only by the substitution
\(t\mapsto qt\).  The numerator thereby becomes a nonnegative combination
of flag \(h\)-numbers of a lower-semimodular lattice, and a classical
labelling of the dual lattice gives a positive descent expansion.  To make
the coefficients explicit, we construct signed edge-lexicographic (EL)
labellings of Ferrers-cell posets with nonincreasing boundaries.  They
give ballot-word formulas in types \(A\), \(B\), and \(C\), which prove
an observation of Chapoton in type \(B\), and a coloured path formula in
type \(D\).  For trapezoidal boundaries, one-boundary major-index
enumerators evaluate the path sums in Gaussian-binomial form.

This work used the generative-AI models OpenAI GPT-5.6 Sol,
Anthropic Claude Fable 5, xAI Grok 4.6, OpenAI GPT-6
Astra, OpenAI GPT-6.1 Sol, Anthropic Claude Fable 5.1, and Anthropic Claude
Opus 5.5.  Further details appear in the disclosure at the end of the paper.
\end{abstract}

\end{frontmatter}

\section{Introduction}
\label{sec:introduction}

Chapoton introduced a \(q\)-analogue of the zeta polynomial of a finite
poset, in which each weak chain is weighted by the sum of the ranks of its
elements \cite{chapoton2024}.  The generating series of these weighted
counts is rational.  Written over a fixed product denominator, it has a
polynomial numerator in \(q\) and \(t\), whose coefficients can be
negative.  When the poset is bounded, graded and admits an
\(R\)-labelling, the numerator is a nonnegative combination of flag
\(h\)-numbers, and hence has nonnegative coefficients
\cite[Proposition~5.2]{chapoton2024}.  Positive-root posets of rank
greater than one are not bounded: they have a highest root but several
minimal roots.  Nevertheless, in type \(B\) Chapoton observed that the
coefficients appear to be nonnegative \(q\)-analogues of the type-\(B\)
Narayana numbers \cite[Section~5.2]{chapoton2024}.

Our starting point is that adjoining a minimum to a graded poset with a
unique maximum changes the numerator only by the substitution
\(t\mapsto qt\) (\cref{lem:minimum-shift}).  Consequently the original
numerator is a combination of the flag \(h\)-numbers of the augmented
poset with nonnegative powers of \(q\)
(\cref{prop:shifted-flag-expansion}).  We obtain these flag numbers in two
ways.  For all irreducible crystallographic root systems, a classical
labelling of the dual lattice proves positivity without case analysis.
For Ferrers-cell posets, which include the root posets of types \(A\),
\(B\) and \(C\), and for a doubled variant in type \(D\), signed cell
labels give the coefficients explicitly.  In these path formulas the
same shift accounts for the contribution of the initial atom.

\subsection{The root-system theorem}

Let \(\Phi\) be a finite reduced irreducible crystallographic root system,
with positive roots \(\PhiPlus\), simple roots \(\Delta\), and highest
root \(\theta\).  Put \(N=\height(\theta)\).  The root order is
\[
  \alpha\leq\beta
  \quad\Longleftrightarrow\quad
  \beta-\alpha\in\sum_{\delta\in\Delta}\NN\delta,
\]
and we use \(\rho(\alpha)=\height(\alpha)-1\).
The numerator is defined by
\begin{equation}
\label{eq:intro-root-numerator}
  \sum_{m\geq0}t^m
    \sum_{\alpha_1\leq\cdots\leq\alpha_m}
      q^{\sum_j\rho(\alpha_j)}
  =\frac{\HH_{\PhiPlus,\rho}(q,t)}
         {\prod_{a=0}^{N-1}(1-q^at)},
\end{equation}
where the empty chain contributes \(1\).  This is Chapoton's fixed
denominator, rather than the denominator of a reduced rational fraction.

\begin{theorem}
\label{thm:all-roots}
For every finite reduced irreducible crystallographic root system
\(\Phi\),
\[
  \HH_{\PhiPlus,\height-1}(q,t)\in\NN[q,t].
\]
Moreover, let \(L_\Phi=\PhiPlus\sqcup\{0\}\), with rank equal to root
height and with \(0\) below every positive root.  For every interval
\(I=[x,y]\) in \(L_\Phi\), graded by
\(\rho_I(z)=\height(z)-\height(x)\), with \(\height(0)=0\),
\[
  \HH_{I,\rho_I}(q,t)\in\NN[q,t].
\]
\end{theorem}

The proof, given in \cref{sec:root-positivity}, is uniform and short once
the minimum-adjunction identity is available.  Panyushev's join theorem
\cite[Appendix~A]{panyushev2014} and covering criterion
\cite[Proposition~2.3(i)]{panyushev2020} make \(L_\Phi\) a
lower-semimodular lattice.  Its dual is therefore upper semimodular and
admits Stanley's classical labelling \cite[Example~3.14.5]{stanley2012}.
Reversing the internal ranks and undoing the minimum adjunction gives the
weighted descent formula of \cref{thm:root-descent}.  Thus the theorem
follows from known lattice structure and flag enumeration; no
computation in individual types is needed.

\subsection{Explicit formulas in types \texorpdfstring{\(B\) and \(D\)}{B and D}}

For the type-\(B_r\) root poset \(P_r\), put \(m=2r-2\).  A word in
\(\{L,R\}\) is ballot if every prefix contains at least as many
\(R\)'s as \(L\)'s.  Set
\[
  \mathcal B_r
  =\{w\in\{L,R\}^{m}:\text{the reversal of \(w\) is ballot}\},
\]
and, for \(w=w_1\cdots w_m\), define
\[
  D(w)=\{j:w_j=L\text{ and }(j=1\text{ or }w_{j-1}=R)\}.
\]
The set \(D(w)\) records the starts of the left runs.  Under the cell
labelling, a right move followed by a left move produces a descent;
an initial left move also produces one, through comparison with the
atom label.  The descent at rank \(j\) has weight \(q^jt\) in the
augmented lattice.  Undoing minimum adjunction replaces this by
\(q^{j-1}t\), which explains the exponent in the following formula.
For \(r=1\), both the word and its descent set are empty.

\begin{theorem}
\label{thm:main}
For every integer \(r\geq1\), with \(\rho=\height-1\),
\[
  \HH_{P_r,\rho}(q,t)
  =\sum_{w\in\mathcal B_r}
    q^{\sum_{j\in D(w)}(j-1)}t^{|D(w)|}.
\]
\end{theorem}

The proof is given in \cref{sec:main-proof}.

The root posets of types \(B_r\) and \(C_r\) are isomorphic
\cite[Section~2.1]{defant2021}; the coordinates in
\cref{lem:cell-model-C} preserve root height, so the same weighted
formula holds in type \(C_r\).  The formula proves Chapoton's type-\(B\)
observation, and at \(q=1\) it gives
\[
  [t^k]\HH_{P_r,\rho}(1,t)=\binom{r-1}{k}^2.
\]
Thus \(\HH_{P_r,\rho}(1,t)\) is the type-\(B\) Narayana polynomial
\(W_{r-1}(t)\) of \cite[Equation~(1.1)]{chen2010qlog}, whose variable
\(q\) is our \(t\); note that the Narayana parameter is \(r-1\), one less
than the Lie rank.
For example,
\[
  \HH_{P_1,\rho}=1,\qquad \HH_{P_2,\rho}=1+t,\qquad
  \HH_{P_3,\rho}=1+(2+q+q^2)t+q^2t^2.
\]
The formula also determines the sharp degree and unique leading monomial
in each nonzero \(t\)-slice (\cref{cor:slice-degree}).

Type \(D_r\) requires a different cell model: the balanced
type-\(B_{r-1}\) shape with a doubled middle column.  Marsh and Williams
used this geometry to count ordinary multichains
\cite[Section~4.3]{marshwilliams2022}.  We label the two coloured lifts
of each balanced path, keeping track of heights.  The descent sets of the
two lifts differ exactly when the move following the entry into the
middle column is a right move, and this local correction gives the
type-\(D\) formula of \cref{thm:type-D-path}.

\subsection{Ferrers boundaries and trapezoids}

For a nonincreasing sequence \(b_1\geq\cdots\geq b_r\geq r\), let
\[
  F_{\mathbf b}=\{(i,c):1\leq i\leq r,\ i\leq c\leq b_i\},
  \qquad \rho(i,c)=c-i,
\]
ordered by \((i,c)\leq(j,d)\) if \(i\geq j\) and \(c\leq d\).
The boundaries \(b_i=r\) and \(b_i=2r-i\) give the root posets of types
\(A_r\) and \(B_r/C_r\).  The construction is not confined to root
posets: signed endpoint labels give a positive path expansion for every
such boundary (\cref{thm:ferrers-family}).  Restricting to intervals and
keeping the full descent set give local and multivariate versions, which
can be computed by a transfer-matrix recursion
(\cref{prop:ferrers-multivariate-refinement,prop:ferrers-transfer-matrix}).
The path description also makes shapes easy to compare: if the top width
is fixed, enlarging the boundary adds paths without changing their
weights, so the numerator increases coefficientwise
(\cref{cor:ferrers-shape-monotonicity}).

For the trapezoids \(b_i=s-i\) with \(s\geq2r\), the labelled paths are
counted by the classical one-boundary descent--major-index enumerators
\cite{krattenthaler1993,chu2017qballot}.  Their Gaussian-binomial
formulas evaluate the contribution of each starting row; together with
the correction for the initial atom they give the whole \(q\)-Zeta
numerator, and the same kernels evaluate all intervals.  At \(q=1\) one
obtains
\[
  \HH_{T_{r,s},\rho}(1,t)
  =\sum_{k=0}^{r-1}\binom{r-1}{k}\binom{s-r-1}{k}t^k,
\]
which is also the numerator of the rectangle \([r]\times[s-r]\), in
accordance with the classical equality of the zeta polynomials of
trapezoids and rectangles recalled after
\cref{cor:trapezoid-qone-profile}; with the rank weights kept, the two
numerators already differ for \(T_{3,6}\).  We also record the Jacobi-polynomial form of these
\(q=1\) profiles and its classical consequences for their zeros.

\Cref{sec:preliminaries} fixes conventions, \cref{sec:augmentation}
proves the minimum-adjunction identity and \cref{sec:root-positivity} the
root-system theorem.  \Cref{sec:cell-model,sec:lattice,sec:main-proof}
treat types \(B\) and \(C\), \cref{sec:type-D} treats type \(D\), and
\cref{sec:trapezoid-family,sec:ferrers-family} treat trapezoidal and
general Ferrers boundaries; type \(A\) appears in
\cref{sec:ferrers-family}.  \Cref{sec:discussion} discusses related
work, the computer checks and some open questions.
The appendices collect the classical \(q=1\) formulas and their
consequences for zeros.

\section{Preliminaries and conventions}
\label{sec:preliminaries}

\subsection{Weak chains and the \texorpdfstring{\(q\)}{q}-Zeta numerator}

Let \(P\) be a finite graded poset with a unique maximum and rank function
\(\rho\colon P\to\{0,\ldots,H\}\).  For \(m\geq 0\), set
\[
  A_m(P,\rho;q)
  =
  \sum_{x_1\leq\cdots\leq x_m}
  q^{\rho(x_1)+\cdots+\rho(x_m)},
\]
where the sum for \(m=0\) consists of the empty chain and equals \(1\).
Write \([n]_q=1+q+\cdots+q^{n-1}\) for \(n\geq1\), and
\([0]_q=0\).  Let \(\mathsf Z_{P,\rho}\) denote Chapoton's
\(q\)-Zeta polynomial.  His Lemma~1.1 identifies
\(A_m(P,\rho;q)=\mathsf Z_{P,\rho}([m+1]_q)\) for \(m\geq1\).
By his Lemma~2.2, the same identity holds for \(m=0\): a poset
with a unique maximum has order complex of Euler characteristic one
\cite[Lemmas~1.1 and~2.2]{chapoton2024}.

\begin{definition}
\label{def:qzeta-series}
The weak-chain \(q\)-Zeta series is
\[
  \FF_{P,\rho}(q,t)
  =
  \sum_{m\geq 0} A_m(P,\rho;q)t^m.
\]
In the normalization of \cite{chapoton2024}, its numerator is the polynomial
\(\HH_{P,\rho}(q,t)\) determined by
\begin{equation}
\label{eq:numerator-definition}
  \FF_{P,\rho}(q,t)
  =
  \frac{\HH_{P,\rho}(q,t)}
       {\prod_{a=0}^{H}(1-q^a t)}.
\end{equation}
\end{definition}

We use \(\NN=\{0,1,2,\ldots\}\).  Thus
\(\HH_{P,\rho}\in\NN[q,t]\) means coefficientwise nonnegativity, not merely
nonnegativity after specializing \(q\) or \(t\).

\subsection{Flag numbers and
  \texorpdfstring{\(R\)}{R}-labellings}

Let \(Q\) be a bounded graded poset of ranks \(0,\ldots,H+1\).
For \(S\subseteq[H]=\{1,\ldots,H\}\), let
\(\alpha_Q(S)\) be the number of chains containing both endpoints whose
internal set of ranks is exactly \(S\).  Its flag \(h\)-number is
\begin{equation}
\label{eq:flag-h-definition}
  \beta_Q(S)
  =
  \sum_{T\subseteq S}(-1)^{|S|-|T|}\alpha_Q(T).
\end{equation}
Write
\[
  \qsum{S}=\sum_{s\in S}s.
\]

In this paper, an \(R\)-labelling uses a totally ordered label set and has
a unique strictly increasing maximal chain in every interval.  It is an
\emph{EL-labelling} if that chain is also lexicographically first among the
maximal chains of the interval
\cite[Definition~2.1 and Proposition~2.5]{bjorner1980}.  For a chain with label
sequence \(\lambda_1,\ldots,\lambda_{H+1}\), its descent set is
\(\{p\in[H]:\lambda_p\geq\lambda_{p+1}\}\), the complement of its strict
ascents.  Along each chain of the explicit signed labellings below the labels
are distinct, as are those of the dual root-poset labelling in
\cref{sec:root-positivity}, so these descents are precisely the strict decreases.

We will use the following two standard flag identities.  The lattice and
labelling results needed for the uniform root-poset theorem are recalled
in \cref{sec:root-positivity}.

\begin{theorem}[Chapoton's flag expansion]
\label{thm:flag-expansion}For a bounded graded poset \(Q\) of ranks \(0,\ldots,H+1\), equipped with
its rank function,
\begin{equation}
\label{eq:bounded-flag-expansion}
  \HH_{Q,\rank}(q,t)
  =
  \sum_{S\subseteq[H]}
  \beta_Q(S)q^{\qsum{S}}t^{|S|}.
\end{equation}
\end{theorem}

\begin{theorem}[\(R\)-labellings and descents]
\label{thm:r-labelling}
If a bounded graded poset \(Q\) has an \(R\)-labelling, then
\begin{equation}
\label{eq:descent-count}
  \beta_Q(S)
  =
  \#\{\text{maximal chains of \(Q\) with descent set \(S\)}\}
\end{equation}
for every set \(S\) of internal ranks.  In particular,
\(\beta_Q(S)\geq 0\).
\end{theorem}

The assertions in \cref{thm:flag-expansion,thm:r-labelling} are
Chapoton's Theorems~A.2 and~A.1, written in the rank indexing of the
present paper: a bounded poset of ranks \(0,\ldots,H+1\) has internal
ranks \([H]\), which is Chapoton's index set \(\{1,\ldots,H'-1\}\) with
\(H'=H+1\) \cite[Section~5 and Appendix~A]{chapoton2024}.
The descent machinery originates in
Stanley \cite[Theorem~3.1 and Corollary~3.2]{stanley1974};
the EL formulation is due to Bj\"orner
\cite[Definition~2.1, Proposition~2.5, and Theorem~2.3]{bjorner1980}.

\subsection{Type-\texorpdfstring{\(B\)}{B} roots}

In \(\mathbb{R}^r\), with standard basis \(e_1,\ldots,e_r\), we use the
simple roots
\[
  \alpha_i=e_i-e_{i+1}\quad(1\leq i<r),
  \qquad
  \alpha_r=e_r.
\]
The positive roots are
\[
  e_i-e_j\ (i<j),\qquad e_i,\qquad e_i+e_j\ (i<j).
\]
For positive roots \(\alpha,\beta\), the root-poset order is
\[
  \alpha\leq\beta
  \quad\Longleftrightarrow\quad
  \beta-\alpha\in\sum_{k=1}^{r}\NN\alpha_k.
\]

\section{Adjoining a minimum}
\label{sec:augmentation}

The next argument is valid for any finite graded poset with a unique
maximum.

\begin{lemma}
\label{lem:minimum-shift}
Let \(P\) have ranks \(0,\ldots,H\) and a unique maximum.  Adjoin a new
minimum \(\bottom\) and denote the resulting bounded poset by \(\wideP\).
Give it the shifted rank
\[
  \widehat{\rho}(\bottom)=0,
  \qquad
  \widehat{\rho}(x)=\rho(x)+1\quad(x\in P).
\]
Then
\begin{align}
  \FF_{\wideP,\widehat{\rho}}(q,t)
    &=\frac{\FF_{P,\rho}(q,qt)}{1-t},
    \label{eq:series-shift}\\
  \HH_{\wideP,\widehat{\rho}}(q,t)
    &=\HH_{P,\rho}(q,qt).
    \label{eq:numerator-shift}
\end{align}
\end{lemma}

\begin{proof}
A weak chain of total length \(m\) in \(\wideP\) has a unique initial block
of \(m-\ell\) copies of \(\bottom\), followed by a weak chain of length
\(\ell\) in \(P\).  Shifting every old rank by one multiplies the weight of
that suffix by \(q^\ell\) (this is the effect of a height shift in
\cite[Lemma~2.6]{chapoton2024}).  Therefore
\[
  A_m(\wideP,\widehat{\rho};q)
  =
  \sum_{\ell=0}^{m}q^\ell A_\ell(P,\rho;q).
\]
Summing in \(m\) gives \cref{eq:series-shift}.

Using \cref{eq:numerator-definition}, the denominator on the right-hand side
of \cref{eq:series-shift} is
\[
  (1-t)\prod_{a=0}^{H}(1-q^a(qt))
  =
  \prod_{j=0}^{H+1}(1-q^jt),
\]
which is exactly the prescribed denominator for \(\widehat{\rho}\).
Both sides have now been written over the same fixed denominator, so
comparison of numerators proves \cref{eq:numerator-shift}; no reduced-fraction
assumption is required.
\end{proof}

\begin{proposition}
\label{prop:shifted-flag-expansion}
With the hypotheses of \cref{lem:minimum-shift},
\begin{equation}
\label{eq:shifted-flag-expansion}
  \HH_{P,\rho}(q,t)
  =
  \sum_{S\subseteq[H]}
  \beta_{\wideP}(S)
  q^{\qsum{S}-|S|}t^{|S|}.
\end{equation}
In particular, every exponent of \(q\) on the right is nonnegative.
\end{proposition}

\begin{proof}
Apply \cref{thm:flag-expansion} to \(\wideP\), then use
\cref{eq:numerator-shift} and replace \(t\) by \(t/q\).  Since every
\(s\in S\) satisfies \(s\geq 1\),
\[
  \qsum{S}-|S|=\sum_{s\in S}(s-1)\geq 0.
\]
\end{proof}

\begin{corollary}[\(R\)-labelling positivity criterion]
\label{cor:r-label-positivity}
If the minimum augmentation \(\widehat P\) of a finite graded poset \(P\) with
unique maximum admits an \(R\)-labelling, then
\[
  \HH_{P,\rho}(q,t)\in\NN[q,t].
\]
\end{corollary}

\begin{proof}
Combine \cref{prop:shifted-flag-expansion} with
\cref{thm:r-labelling}: each coefficient is a sum of numbers of maximal
chains with prescribed descent set.
\end{proof}

\section{Positivity for irreducible root posets}
\label{sec:root-positivity}

Let \(\Phi\) be a finite reduced irreducible crystallographic root system,
with simple roots \(\Delta\), positive roots \(\PhiPlus\), and highest root
\(\theta\).  Write
\[
  L_\Phi=\PhiPlus\sqcup\{0\},\qquad
  \widehat\rho(\gamma)=\height(\gamma),\qquad
  \widehat\rho(0)=0,\qquad N=\height(\theta).
\]
The order is coefficientwise comparison in the simple-root basis, with
\(0\) below every positive root.  Root-poset covers subtract one simple
root; in particular, \(L_\Phi\) is graded by height
\cite[Section~1]{panyushev2006}.

We first recall the root-poset structure needed for the flag argument.
Panyushev proved that every pair of positive roots has a join
\cite[Appendix~A, Theorem~A.1]{panyushev2014}.  His covering criterion
\cite[Proposition~2.3(i)]{panyushev2020} gives lower semimodularity after
adjoining zero.  We include a self-contained proof of that deduction,
using the following root-diamond lemma of Panyushev and R\"ohrle
\cite[Lemma~3.1]{panyushevrohrle2001}.

\begin{lemma}
\label{lem:root-difference}
If \(u,v\in\Phi\) are nonproportional and \((u,v)>0\), then
\(u-v\in\Phi\).
\end{lemma}

\begin{proof}
The crystallographic integers
\[
  a=\frac{2(u,v)}{(v,v)},\qquad
  b=\frac{2(u,v)}{(u,u)}
\]
are positive.  Strict Cauchy--Schwarz gives \(ab<4\), so at least one
of \(a,b\) equals one.  If \(a=1\), reflection in \(v\) gives
\(u-v\in\Phi\).  If \(b=1\), reflection in \(u\) gives
\(v-u\in\Phi\), and we negate this root.
\end{proof}

\begin{lemma}[Root diamonds]
\label{lem:root-diamond}
Suppose that \(\gamma\in\PhiPlus\) covers distinct positive roots
\(\gamma-a\) and \(\gamma-b\), where \(a,b\in\Delta\).
Then \(\gamma-a-b\in\PhiPlus\cup\{0\}\).
\end{lemma}

\begin{proof}
If \(\gamma=a+b\), the conclusion is immediate, so assume
\(\gamma\neq a+b\).  The coefficients of \(a\) and \(b\) in \(\gamma\) are
at least one, so all simple-root coefficients of \(\gamma-a-b\) are
nonnegative; it therefore suffices to show that \(\gamma-a-b\) is a root.
Suppose it is not.  The roots \(\gamma-a\) and \(b\) are
nonproportional, since by reducedness the only positive root
proportional to \(b\) is \(b\) itself, and \(\gamma-a=b\) is excluded.
Likewise \(\gamma-b\) and \(a\) are nonproportional.  By
\cref{lem:root-difference},
\[
  (\gamma-a,b)\leq0,\qquad (\gamma-b,a)\leq0.
\]
Consequently
\[
\begin{aligned}
  (\gamma-a,\gamma-b)
  &=(\gamma,\gamma)-(\gamma,a)-(\gamma,b)+(a,b)\\
  &\geq(\gamma,\gamma)-(a,b)>0,
\end{aligned}
\]
since distinct simple roots have nonpositive inner product.
The two positive roots \(\gamma-a\) and \(\gamma-b\) are distinct, hence
nonproportional, so another application of \cref{lem:root-difference}
would make \(b-a\) a root.  This is impossible because its simple-root
coefficients have both signs.  Hence \(\gamma-a-b\) is a root, and by the
first paragraph it is positive.
\end{proof}

\begin{proposition}
\label{prop:root-lattice}
The augmented root poset \(L_\Phi\) is a lower-semimodular lattice.
\end{proposition}

\begin{proof}
The joins of positive roots exist by Panyushev's theorem.  Adding \(0\)
makes this finite join-semilattice a lattice: the meet of two elements is
the join of their common lower bounds, a nonempty set containing \(0\).

Suppose that \(\gamma\) covers two distinct elements \(u,v\) of
\(L_\Phi\), so that \(\gamma=u\vee v\).  Neither is zero, since a root
covering zero is simple and cannot cover a second element.  Write \(u=\gamma-a\) and
\(v=\gamma-b\) with distinct \(a,b\in\Delta\).  By
\cref{lem:root-diamond}, \(z=\gamma-a-b\) belongs to \(L_\Phi\).
It is the coefficientwise minimum of \(u,v\), so it is their meet.
Both \(u\) and \(v\) cover \(z\), because their heights exceed its
height by one.  Applying the dual of the semimodular covering criterion
\cite[Proposition~3.3.2]{stanley2012} proves the assertion.
\end{proof}

\subsection{Dual labels and the weighted descent formula}

For completeness we give the classical labelling with the precise
strictness convention used here.  Let \(M\) be a finite upper-semimodular
lattice.  Enumerate its join-irreducibles \(a_1,\ldots,a_k\) in an
order-preserving way, so \(a_i<a_j\) implies \(i<j\).  For a cover
\(x\lessdot y\), put
\begin{equation}
\label{eq:semimodular-r-label}
  \lambda(x,y)=\min\{j:x\vee a_j=y\}.
\end{equation}
This is Stanley's labelling
\cite[Example~3.14.5, Equation~(3.57)]{stanley2012}.

\begin{lemma}
\label{lem:semimodular-strict-r}
The labelling in \cref{eq:semimodular-r-label} is an \(R\)-labelling.
Its labels along each maximal chain are distinct, and its increasing chain
in every interval is lexicographically first.
\end{lemma}

\begin{proof}
Every element of a finite lattice is a join of join-irreducibles below it.
For a cover \(x\lessdot y\), some \(a_j\leq y\) is not below \(x\),
and then \(x<x\vee a_j\leq y\); thus the label is defined.  A label
\(j\) can occur only at the edge on which \(a_j\) first becomes less
than or equal to the chain element.  Hence labels cannot repeat.

Fix \(x<y\), and let \(j\) be the smallest index with
\(a_j\leq y\) and \(a_j\nleq x\).  If \(a_j^-\) is the unique lower
cover of \(a_j\), all join-irreducibles strictly below \(a_j\),
equivalently below \(a_j^-\), have index less than \(j\) and lie below
\(x\).  Thus
\(a_j^-\leq x\) and \(x\wedge a_j=a_j^-\).
Upper semimodularity now implies that \(x\vee a_j\) covers \(x\).
This edge has label \(j\), since no smaller-index join-irreducible below
\(y\) remains outside \(x\).  Repeating this choice constructs a strictly
increasing chain to \(y\).

Every chain from \(x\) to \(y\) must first acquire \(a_j\) at some
cover, and that cover has label \(j\).  If its first label were larger
than \(j\), it could not be increasing.  Its first label cannot be
smaller, by the choice of \(j\).  Therefore an increasing chain must
start with \(x\vee a_j\), and induction gives uniqueness.  At each
step this edge has the smallest possible label.  It is also determined
by that label, so the same induction gives the lexicographically first
chain.
\end{proof}

Apply this labelling to \(M=L_\Phi^*\), which is upper semimodular by
\cref{prop:root-lattice}.  A maximal chain of \(M\) is a downward root
chain
\[
  C:\quad\theta=x_0\gtrdot x_1\gtrdot\cdots\gtrdot x_N=0.
\]
Write \(D(C)=\{i\in[N-1]:\lambda_i>\lambda_{i+1}\}\), where
\(\lambda_i\) labels the \(i\)-th edge in this downward direction.

\begin{theorem}[Root-poset descent formula]
\label{thm:root-descent}
With \(\rho=\height-1\) on \(\PhiPlus\),
\begin{equation}
\label{eq:root-descent}
  \HH_{\PhiPlus,\rho}(q,t)
  =\sum_C q^{\sum_{i\in D(C)}(N-i-1)}t^{|D(C)|}.
\end{equation}
For any interval \(I=[x,y]\) in \(L_\Phi\), put
\(d=\widehat\rho(y)-\widehat\rho(x)\) and use the local rank
\(\rho_I(z)=\widehat\rho(z)-\widehat\rho(x)\).  Number the edges of each
downward maximal chain of \(I\) from \(1\) to \(d\), and let \(D(C)\)
denote its descent positions in this local numbering.  Restricting the
same dual labelling gives
\begin{equation}
\label{eq:root-interval-descent}
  \HH_{I,\rho_I}(q,t)
  =\sum_C q^{\sum_{i\in D(C)}(d-i)}t^{|D(C)|}.
\end{equation}
Both polynomials belong to \(\NN[q,t]\).
\end{theorem}

\begin{proof}
Reversing chains reverses flag ranks.  If
\(N-S=\{N-s:s\in S\}\), then
\[
  \beta_{L_\Phi}(S)=\beta_{L_\Phi^*}(N-S).
\]
By \cref{thm:r-labelling,lem:semimodular-strict-r}, the right-hand side
counts downward maximal chains with descent set \(N-S\).  Substitute
this in \cref{prop:shifted-flag-expansion}.  The exponent
\(\sum_{s\in S}(s-1)\) becomes
\(\sum_{i\in D(C)}(N-i-1)\), proving \cref{eq:root-descent}.

For a singleton interval the numerator is \(1\) directly from the
weak-chain definition.  Otherwise \(d\geq1\), and the restriction of
the labelling to \(I^*\) is still an \(R\)-labelling.
Apply the bounded flag expansion directly to \(I\), reversing its
internal ranks \(1,\ldots,d-1\).  There is no minimum-adjunction shift
in this local calculation, so the exponent is
\(\sum_{i\in D(C)}(d-i)\), proving \cref{eq:root-interval-descent}.
All exponents are nonnegative.  (For \(N=1\) the unique maximal chain has
no descents, and both sides of \cref{eq:root-descent} equal \(1\).)
\end{proof}

\begin{proof}[Proof of \cref{thm:all-roots}]
The global rank range is \(0,\ldots,N-1\), so its fixed denominator is
\(\prod_{a=0}^{N-1}(1-q^at)\).  The two assertions follow from
\cref{thm:root-descent}.
\end{proof}

\begin{remark}
Irreducibility cannot be omitted.  For the disjoint union of the
\(A_2\) and \(A_1\) positive-root posets, with height minus one,
Chapoton's evaluation series is
\[
  \frac{1+t}{(1-t)(1-qt)}+\frac{1}{1-t}
  =\frac{2+(1-q)t}{(1-t)(1-qt)},
\]
whose fixed numerator has a negative coefficient.  Indeed, the
\(q\)-Zeta polynomial is additive under disjoint unions
\cite[Lemma~2.1]{chapoton2024}, and its value at \([1]_q\) here is
the Euler characteristic \(2\) \cite[Lemma~2.2]{chapoton2024}.
Thus the series has constant term \(2\), not the single empty-chain
term used in \cref{def:qzeta-series}.  The latter convention would
instead give numerator \(1+2t-qt^2\), which also has a negative
coefficient.  The two conventions agree for the posets with a unique
maximum considered in our theorems.
\end{remark}

\section{A common cell model for the type-\texorpdfstring{\(B\)}{B} and
  \texorpdfstring{\(C\)}{C} root posets}
\label{sec:cell-model}

Define the trapezoidal set
\[
  C_r=\{(i,c):1\leq i\leq r,\ i\leq c\leq 2r-i\}.
\]
Here \(C_r\) denotes a cell poset; the phrase ``type \(C_r\)'' refers to
the root system, and its positive-root poset is written \(\Phi^+(C_r)\).
Give it the order
\begin{equation}
\label{eq:cell-order}
  (i,c)\leq(j,d)
  \quad\Longleftrightarrow\quad
  i\geq j\ \text{and}\ c\leq d.
\end{equation}

\begin{lemma}
\label{lem:cell-model}
The map
\begin{align}
  e_i-e_j&\longmapsto(i,j-1),\nonumber\\
  e_i&\longmapsto(i,r),\label{eq:root-cell-map}\\
  e_i+e_j&\longmapsto(i,2r-j+1)\nonumber
\end{align}
is a rank-preserving isomorphism from the positive-root poset
\(P_r=\PhiPlus(B_r)\) to \(C_r\).  Under this map,
\[
  \rank(i,c)=c-i.
\]
\end{lemma}

\begin{proof}
For fixed \(i\), the three root families in \cref{eq:root-cell-map} fill,
respectively, columns \(i,\ldots,r-1\), column \(r\), and columns
\(r+1,\ldots,2r-i\).  The map is therefore a bijection onto \(C_r\).

We verify the order directly in simple-root coordinates.  If \(v_{i,c}(k)\)
is the coefficient of \(\alpha_k\) in the root represented by \((i,c)\),
then for \(c\leq r\),
\[
  v_{i,c}(k)=
  \begin{cases}
    1,&i\leq k\leq c,\\
    0,&\text{otherwise},
  \end{cases}
\]
whereas for \(c>r\), writing \(m=2r-c+1\),
\[
  v_{i,c}(k)=
  \begin{cases}
    0,&k<i,\\
    1,&i\leq k<m,\\
    2,&m\leq k\leq r.
  \end{cases}
\]
The cell condition gives \(m\geq i+1\).

Suppose \(i\geq j\) and \(c\leq d\).  If both columns are at most \(r\),
the support interval for \(v_{i,c}\) is contained in that for \(v_{j,d}\).
If \(c\leq r<d\), the latter vector is at least one on the entire interval
from \(j\) to \(r\).  If \(r<c\leq d\), both the support of the positive
entries and the terminal interval of entries equal to two for \(v_{i,c}\)
are contained in the corresponding intervals for \(v_{j,d}\).  Hence
\(v_{i,c}\leq v_{j,d}\) coordinatewise.

Conversely, assume \(v_{i,c}\leq v_{j,d}\).  Comparing the first nonzero
coordinate gives \(i\geq j\).  If \(c\leq r\) and \(d<r\), then \(c>d\)
would give \(v_{i,c}(c)=1>0=v_{j,d}(c)\); if \(d\geq r\), then
\(c\leq d\) is automatic.  If \(c>r\), the coefficient at \(\alpha_r\)
forces \(d>r\).  Were \(c>d\), the terminal interval of entries equal to
two for \(v_{i,c}\) would begin strictly earlier than that for \(v_{j,d}\),
again contradicting coordinatewise domination.  Thus \(c\leq d\), proving
\cref{eq:cell-order}.

Finally, the sum of the simple-root coefficients is \(c-i+1\) in both
column regions, so root height minus one is \(c-i\).
\end{proof}

\begin{lemma}[Type-$C$ companion cell model]
\label{lem:cell-model-C}
Let $P_r^{C}=\Phi^+(C_r)$ be the standard positive-root poset of type
$C_r$, with simple roots $\alpha_i=e_i-e_{i+1}$ for $i<r$ and
$\alpha_r=2e_r$.  The map
\begin{align}
  e_i-e_j&\longmapsto(i,j-1),\nonumber\\
  e_i+e_j&\longmapsto(i,2r-j),\label{eq:root-cell-map-C}\\
  2e_i&\longmapsto(i,2r-i)\nonumber
\end{align}
is a rank-preserving isomorphism from $P_r^{C}$ to the same trapezoidal
cell poset $C_r$, where the rank is again $c-i$.
\end{lemma}

\begin{proof}
For fixed $i$, the difference roots fill columns $i,\ldots,r-1$, the
roots $e_i+e_j$ fill columns $r,\ldots,2r-i-1$, and $2e_i$ fills the
last column $2r-i$.  Thus the map is bijective.  In simple-root
coordinates, a cell with $c\leq r$ has coefficients one on
$i\leq k\leq c$ and zero elsewhere.  For $c>r$, put $m=2r-c$; its
coefficient vector has value zero for $k<i$, value one for $i\leq k<m$,
value two for $m\leq k<r$, and value one at $k=r$.

Suppose $i\geq j$ and $c\leq d$.  If $d\leq r$, containment of the
intervals of ones is immediate.  If $c\leq r<d$, the vector for $(j,d)$
is at least one on every coordinate where the vector for $(i,c)$ is
nonzero, and it is one at the final coordinate.  If $r<c\leq d$, then
$2r-d\leq2r-c$; the support of the ones starts no later and the terminal
interval of twos starts no later, so coordinatewise domination again holds.
Conversely, coordinatewise domination first gives $i\geq j$ from the first
nonzero coordinate.  The case $c\leq r$ and $d<r$ then gives $c\leq d$ by
looking at coordinate $c$; if $d\geq r$ the inequality is automatic.  If
$c>r$, then $d\leq r$ is impossible because the left vector has a two at
coordinate $r-1$ whereas the right vector has at most one.  Thus $d>r$,
and $c>d$ would make the two-interval of the left vector start strictly
earlier, contradicting domination.  Hence $c\leq d$.

Finally, the sum of the simple-root coefficients is $c-i+1$, so height
minus one is $c-i$.
\end{proof}

\begin{corollary}
\label{cor:cell-ranks}
The poset \(P_r\) has ranks \(0,\ldots,H\), where \(H=2r-2\), and has a
unique maximum, represented by \((1,2r-1)\).
\end{corollary}

\begin{proof}
The minimum possible value of \(c-i\) is zero.  The largest is
\((2r-1)-1=2r-2\), attained only at \((1,2r-1)\).
\end{proof}

\section{The augmented type-\texorpdfstring{\(B\)}{B} lattice}
\label{sec:lattice}

Let
\[
  \wideL_r=C_r\sqcup\{\bottom\},
  \qquad
  \widehat{\rho}(\bottom)=0,
  \qquad
  \widehat{\rho}(i,c)=c-i+1.
\]
The new element \(\bottom\) lies below every cell.

\begin{proposition}
\label{prop:lattice-operations}
The poset \(\wideL_r\) is a bounded graded lattice.  For nonbottom cells
\(x=(i,c)\) and \(y=(j,d)\),
\begin{align}
  x\vee y
    &=(\min(i,j),\max(c,d)),\label{eq:join}\\
  x\wedge y
    &=
    \begin{cases}
      (\max(i,j),\min(c,d)),
        &\max(i,j)\leq\min(c,d),\\
      \bottom,&\max(i,j)>\min(c,d).
    \end{cases}
    \label{eq:meet}
\end{align}
Its rank function is \(\widehat{\rho}\).
\end{proposition}

\begin{proof}
Put \(I=\min(i,j)\) and \(C=\max(c,d)\).  We have \(I\leq C\), and
\[
  C\leq\max(2r-i,2r-j)=2r-I,
\]
so the cell in \cref{eq:join} always lies in \(C_r\).  The coordinate order
in \cref{eq:cell-order} shows that it is the least common upper bound.

For the proposed nonbottom meet, put \(J=\max(i,j)\) and
\(D=\min(c,d)\).  When \(J\leq D\),
\[
  D\leq\min(2r-i,2r-j)=2r-J,
\]
so \((J,D)\) is a cell and is the greatest common lower bound.  When
\(J>D\), any nonbottom common lower bound would need first coordinate at
least \(J\) and second coordinate at most \(D\), contradicting the defining
cell inequality that the first coordinate is at most the second.  Its meet
is therefore \(\bottom\).

Among cells, a cover changes exactly one of \(-i\) and \(c\) by one while
remaining in \(C_r\); the new minimum is covered precisely by the diagonal
cells \((i,i)\).  Equivalently, if two comparable cells have rank difference
at least two, changing \(i\) downward by one when possible, or increasing
\(c\) by one otherwise, supplies an intermediate cell.  Thus
\(\widehat{\rho}\) is the lattice rank.
\end{proof}

\begin{proposition}
\label{prop:lower-semimodular}
The lattice \(\wideL_r\) is lower semimodular.  Equivalently, its rank is
supermodular:
\begin{equation}
\label{eq:supermodular}
  \widehat{\rho}(x)+\widehat{\rho}(y)
  \leq
  \widehat{\rho}(x\wedge y)+\widehat{\rho}(x\vee y)
\end{equation}
for all \(x,y\in\wideL_r\).
\end{proposition}

\begin{proof}
Pairs involving \(\bottom\) give equality.  For two cells with nonbottom
meet, direct substitution of \cref{eq:join,eq:meet} also gives equality.

Suppose the meet is \(\bottom\).  Interchanging \(x\) and \(y\) if necessary,
assume \(i\geq j\).  The failed cell condition is
\(i>\min(c,d)\).  Since \(c\geq i\), it follows that \(d<i\leq c\).
Consequently \(x\vee y=(j,c)\), and
\[
  \widehat{\rho}(x)+\widehat{\rho}(y)
  -\widehat{\rho}(x\vee y)
  =d-i+1\leq 0.
\]
Because \(\widehat{\rho}(x\wedge y)=0\), this is
\cref{eq:supermodular}.
\end{proof}

\Cref{prop:lower-semimodular} is the type-\(B_r\) instance of
\cref{prop:root-lattice}.  We record the coordinate computation because
it applies verbatim to the trapezoidal and Ferrers boundaries of
\cref{sec:trapezoid-family,sec:ferrers-family}, which are not root
posets in general.  The descent formula, however, uses a different
labelling, adapted to the cell coordinates.

Chapoton labelled the covers of the type-\(B\) root poset by their
direction \cite[Section~5.2]{chapoton2024}.  The signed labels below
refine this idea and also label the covers from the new minimum; the
comparison between an atom label and the first move produces the initial
descent in \cref{thm:main}.

\begin{proposition}
\label{prop:explicit-el}
Label the covers of \(\wideL_r\) by integers, with their usual order, as
follows:
\begin{align}
  \lambda\bigl(\bottom,(i,i)\bigr)&=i,\nonumber\\
  \lambda\bigl((i,c),(i-1,c)\bigr)&=-(i-1),\label{eq:explicit-labels}\\
  \lambda\bigl((i,c),(i,c+1)\bigr)&=c+1.\nonumber
\end{align}
Whenever the displayed pairs are covers, this is an EL-labelling of
\(\wideL_r\).  In particular the labels along any maximal chain are
pairwise distinct, so the strictly increasing chain is also the unique
weakly increasing chain.
\end{proposition}

\begin{proof}
Consider first an interval from \((i,c)\) to \((j,d)\), with neither
endpoint equal to \(\bottom\).  Every maximal chain is a shuffle of the
left-expansion labels
\[
  -(i-1),-(i-2),\ldots,-j
\]
and the right-expansion labels
\[
  c+1,c+2,\ldots,d,
\]
with each displayed list used in its given order.  Both lists are strictly
increasing.  Because every left label is negative and every right label is
positive, the unique increasing shuffle performs all left expansions before
all right expansions.  At every step where both moves are available, it
chooses the smaller label, so it is also lexicographically first.

Now consider an interval from \(\bottom\) to \((j,d)\).  Its unique
increasing maximal chain begins with the atom \((j,j)\), then makes only
right expansions; its labels are
\[
  j,j+1,\ldots,d.
\]
A chain beginning at an atom \((i,i)\) with \(i>j\) must eventually make a
left expansion, producing a positive-to-negative descent.  Moreover, the
initial label \(j\) is the smallest label of an atom below \((j,d)\).
Hence the displayed chain is uniquely increasing and lexicographically
first.  Degenerate intervals are unique chains, so these cases cover every
interval.
\end{proof}

\section{EL descents and the ballot-word formula}
\label{sec:main-proof}

We now identify the maximal chains and descents of the explicit labelling in
\cref{prop:explicit-el}.  A cover after the initial atom is called an
\(L\)-move if it decreases the first cell coordinate and an \(R\)-move if
it increases the second.

\begin{lemma}
\label{lem:chain-words}
Reading the \(L\)- and \(R\)-moves gives a bijection from the maximal chains
of \(\wideL_r\) to \(\mathcal B_r\).  Under this bijection, the descent set
of the label sequence is \(D(w)\).
\end{lemma}

\begin{proof}
A maximal chain begins at an atom \((i,i)\) and ends at
\((1,2r-1)\).  After the initial edge, it has \(i-1\) left moves and
\(2r-1-i\) right moves, hence \(m=2r-2\) moves in total.  If \(L_k\) and
\(R_k\) count the two moves in the first \(k\) letters, the intermediate
cell is
\[
  (i-L_k,i+R_k).
\]
Its right boundary condition is
\begin{equation}
\label{eq:path-boundary}
  R_k-L_k\leq 2(r-i).
\end{equation}
The remaining cell inequalities \(1\leq i-L_k\leq r\) and
\(i-L_k\leq i+R_k\) hold automatically: there are only \(i-1\) letters
\(L\) in the whole word, so the first coordinate decreases from \(i\) to
\(1\) without leaving \(\{1,\ldots,r\}\).  At the end equality holds in
\cref{eq:path-boundary}.  Consequently, \cref{eq:path-boundary} for all
prefixes is equivalent to saying that every suffix of the word has
nonnegative \(R\)-minus-\(L\) balance.  This is exactly the ballot-prefix
condition on the reversed word.

Conversely, let \(w\in\mathcal B_r\), and put
\(b=\#R-\#L\).  The ballot condition gives \(0\leq b\leq 2r-2\), and
\(b\) is even.  Thus \(i=r-b/2\) is an integer in \([1,r]\).  Starting at
\((i,i)\) and following \(w\), the reversed ballot inequalities are exactly
\cref{eq:path-boundary}; the path therefore stays in \(C_r\) and terminates
at \((1,2r-1)\).  This constructs the inverse bijection.

It remains to read the descents.  The maximal chain has \(2r-1\) covers, and
descents occur at internal ranks \(1,\ldots,2r-2\).  The initial label is
the positive integer \(i\).  Consecutive left moves have increasing negative
labels, consecutive right moves have increasing positive labels, and a left
move followed by a right move is also increasing.  Comparing the initial
label with the first move therefore produces a descent at rank \(1\) if and
only if that move is left.  For \(j\geq 2\), comparing the \((j-1)\)-st and
\(j\)-th moves produces a descent at rank \(j\) if and only if a right move
is followed by a left move.  These are exactly the elements of \(D(w)\).
\end{proof}

\begin{proof}[Proof of \cref{thm:main}]
By \cref{lem:cell-model,prop:lattice-operations}, adjoining a minimum to
\(P_r\) gives the bounded graded lattice \(\wideL_r\).  The explicit
EL-labelling of \cref{prop:explicit-el} is an \(R\)-labelling, so
\cref{eq:descent-count,prop:shifted-flag-expansion} give
\[
  \HH_{P_r,\rank}(q,t)
  =
  \sum_{\mathfrak c}
  q^{\qsum{\operatorname{Des}(\mathfrak c)}
      -|\operatorname{Des}(\mathfrak c)|}
  t^{|\operatorname{Des}(\mathfrak c)|},
\]
where the sum is over the maximal chains of \(\wideL_r\).  Apply
\cref{lem:chain-words}: the chains become the words \(w\in\mathcal B_r\),
their descent sets become \(D(w)\), and
\[
  \qsum{D(w)}-|D(w)|  =
  \sum_{j\in D(w)}(j-1).
\]
This is the asserted formula.  Every summand is a monomial with coefficient
one and nonnegative exponents, proving coefficientwise nonnegativity as
well.  For \(r=1\), the unique chain corresponds to the empty word and the
formula gives \(\HH_{P_1,\rank}(q,t)=1\).
\end{proof}

\begin{corollary}[Type-$C$ companion formula]
\label{cor:type-C}
For the standard positive-root poset $P_r^{C}=\Phi^+(C_r)$, graded by
root height minus one and with Chapoton's denominator convention, one has
\[
  \HH_{P_r^{C},\rank}(q,t)
  =\sum_{w\in\mathcal B_r}
    q^{\sum_{j\in D(w)}(j-1)}t^{|D(w)|}.
\]
In particular, its \(q\)-Zeta numerator is coefficientwise nonnegative for
every $r\geq1$.
\end{corollary}

\begin{proof}
By \cref{lem:cell-model-C}, the ranked posets $P_r^{C}$ and $P_r$ are
isomorphic.  The \(q\)-Zeta numerator is defined from rank-weighted weak-chain
enumeration, hence is invariant under a rank-preserving poset isomorphism.
The assertion follows from \cref{thm:main}.
\end{proof}

We record some consequences of \cref{thm:main}.

\begin{corollary}
\label{cor:sharp-degree}
For every \(r\geq1\),
\[
  \HH_{P_r,\rank}(1,1)=\binom{2r-2}{r-1},
  \qquad
  \deg_t\HH_{P_r,\rank}=r-1,
\]
and the coefficient of the highest power of \(t\) is
\[
  [t^{r-1}]\HH_{P_r,\rank}(q,t)=q^{(r-1)(r-2)}.
\]
Moreover, for \(0\leq k\leq r-1\),
\[
  [t^k]\HH_{P_r,\rank}(1,t)=\binom{r-1}{k}^{2}.
\]
\end{corollary}

The following count is needed for the coefficients at \(q=1\); we include
a short proof.

\begin{lemma}\label{lem:ballot-runs}
Let \(a(n,k)\) be the number of words of length \(2n\) whose every prefix
has at least as many \(R\)'s as \(L\)'s and which have exactly \(k\)
\(L\)-runs.  Then
\[
  a(n,k)=\binom{n}{k}^{2}.
\]
\end{lemma}

\begin{proof}
Let \(C(z,u)\) be the generating function of Dyck words, with \(z\)
marking semilength and \(u\) marking \(L\)-runs, using the same convention
that \(R\) is an up-step.  The first-return
decomposition gives
\[
  C=1+zC(C+u-1).
\]
Indeed, in a first-return factor \(R A L B\), the outer pair creates a new
\(L\)-run exactly when \(A\) is empty.  Here the number of \(L\)-runs is
also the number of peaks.  For further Catalan path statistics with the
Narayana distribution, see
\cite[Section~1 and Propositions~2--3]{sulanke1998}.
Every ballot word has a unique
factorization
\[
  D_0 R D_1 R\cdots R D_h,
\]
where the \(D_i\) are Dyck words and \(h\) is the final height.  Taking the
even-length part (so \(h=2j\)) therefore gives the generating function
\[
  E(z,u)=\sum_{n,k}a(n,k)z^n u^k
       =\frac{C(z,u)}{1-zC(z,u)^2}.
\]
Put \(Y=C-1\), so \(Y=z(1+Y)(u+Y)\).  The coefficients of \(E\) are
extracted from this functional equation by Lagrange--Bürmann inversion.
First expand
\[
  E(z,u)=\sum_{j\geq0}z^j C(z,u)^{2j+1}.
\]
For \(s\geq1\), \(p\geq1\), and \(1\leq k\leq s\), inversion applied to
\(Y=z(1+Y)(u+Y)\) gives
\[
\begin{aligned}
 [z^s u^k]C(z,u)^p
 &=\frac{p}{s}[y^{s-1}u^k]
      (1+y)^{p+s-1}(u+y)^s \\
 &=\frac{p}{s}\binom{s}{k}\binom{p+s-1}{k-1}.
\end{aligned}
\]
Consequently, for \(n\geq1\) and \(1\leq k\leq n\), with \(s=n-j\),
\[
 [z^n u^k]E
 =\sum_{s=k}^{n}
   \frac{2n-2s+1}{s}\binom{s}{k}\binom{2n-s}{k-1}.
\]
To evaluate this finite sum, use
\(\binom{s}{k}/s=\binom{s-1}{k-1}/k\), put \(A=s-1\) and
\(B=2n-1-A\), and apply Pascal's identity:
\[
\begin{aligned}
 &\frac{B-A}{k}\binom{A}{k-1}\binom{B}{k-1} \\
 &\quad=\binom{A}{k-1}\binom{B}{k}
       -\binom{A}{k}\binom{B}{k-1} \\
 &\quad=\binom{A+1}{k}\binom{2n-A-1}{k}
       -\binom{A}{k}\binom{2n-A}{k}.
\end{aligned}
\]
The sum therefore telescopes from \(A=k-1\) to \(A=n-1\), and is
\(\binom{n}{k}^{2}\).  For \(k=0\), setting \(u=0\) gives
\(C(z,0)=1\) and \(E(z,0)=1/(1-z)\), so the coefficient is \(1\); for
\(n=0\) the empty word gives the same conclusion.  Hence
\([z^n u^k]E(z,u)=\binom{n}{k}^{2}\) in all cases.
\end{proof}

\begin{proof}[Proof of \cref{cor:sharp-degree}]
Write \(n=r-1\).  Reversal identifies \(\mathcal B_r\) with the ballot
words of length \(2n\).  By the reflection principle, the number ending at
height \(2k\) is
\[
  \binom{2n}{n-k}-\binom{2n}{n-k-1}
  \qquad(0\leq k\leq n),
\]
where \(\binom{2n}{-1}=0\).  Summing over \(k\) telescopes to
\(\binom{2n}{n}\), proving the first identity.

Every reversed ballot word has at most \(n\) letters \(L\), so
\(|D(w)|\leq n\).  Equality forces exactly \(n\) letters \(L\), each
starting a separate run.  For \(n=0\) the unique word is empty.  For
\(n\geq 1\), a nonempty word in \(\mathcal B_r\) ends in \(R\), and the
unique such word with \(n\) isolated letters \(L\) is
\[
  w=(LR)^n.
\]
Its descent positions are \(1,3,\ldots,2n-1\), and hence its \(q\)-exponent
is
\[
  \sum_{a=0}^{n-1}2a=n(n-1).
\]
This proves both the sharp degree and the stated leading coefficient.

It remains to evaluate the coefficients at \(q=1\).  By
\cref{thm:main}, the coefficient of \(t^k\) is the number of ballot words
of length \(2n\) with exactly \(k\) runs of the letter \(L\); reversal
preserves the number of runs.  \cref{lem:ballot-runs} therefore yields
\[
  [t^k]\HH_{P_r,\rank}(1,t)=\binom{r-1}{k}^{2}.
\]
\end{proof}

\begin{corollary}[Sharp $q$-degree in each $t$-slice]
\label{cor:slice-degree}
Put $n=r-1$.  For $0\leq k\leq n$, the $t^k$-coefficient satisfies
\[
  [t^k]\HH_{P_r,\rank}(q,t)
  \in q^{k(k-1)}\NN[q],
  \qquad
  \deg_q [t^k]\HH_{P_r,\rank}(q,t)=k(2n-k-1).
\]
Moreover, the coefficient of the highest power
\(q^{k(2n-k-1)}\) is equal to~$1$.
\end{corollary}

\begin{proof}
Write the elements of $D(w)$ as $j_1<\cdots<j_k$.  Two consecutive
left-run starts are separated by at least one right letter, so
$j_a\geq 2a-1$.  Hence
\[
  \sum_{a=1}^k(j_a-1)\geq \sum_{a=1}^k(2a-2)=k(k-1).
\]
For the upper bound, reverse $w$ and write $m=2n$.  A marked position in
$D(w)$ becomes either a valley, an $L$ followed by an $R$, or a terminal
$L$ in the ballot word.  If these marked positions are
$p_1<\cdots<p_k$, the ballot condition and the initial $R$ imply
$p_a\geq 2a$.  Therefore
\[
  \sum_{a=1}^k(j_a-1)=\sum_{a=1}^k(m-p_a)
  \leq \sum_{a=1}^k(m-2a)=k(2n-k-1).
\]
Equality forces the reversed word to be
\[
  (RL)^kR^{2n-2k},
\]
whose reversal is in $\mathcal B_r$ and has the required run starts.
Thus the upper degree is attained by exactly one word, proving the final
assertion.
\end{proof}

\subsection{The specialization \texorpdfstring{\(q=1\)}{q=1}}
\label{sec:spectral-consequences}

By \cref{cor:sharp-degree}, the specialization \(q=1\) of the
type-\(B_r\) numerator is the type-\(B\) Narayana polynomial
\[
  B_n(t)=\sum_{k=0}^{n}\binom{n}{k}^{2}t^k,\qquad n=r-1.
\]
\ref{app:classical-type-B} records the classical consequences of the
Legendre representation of \(B_n\), including its real-rootedness,
interlacing, \(\gamma\)-positivity, and limiting zero law. These provide a
comparison with the trapezoidal family in \cref{sec:trapezoid-family};
none involves the variable \(q\).

\begin{remark}
\label{rem:three-rank}
The flag expansion also identifies individual coefficients with
rank-selected Euler characteristics.  For instance, let
\(0\leq a<b<c\leq 2r-3\), let \(f_a\) be the number of elements of
\(P_r\) of rank \(a\), let \(e_{ab}\) be the number of comparable pairs
in ranks \(a\) and \(b\), and let \(n_{abc}\) be the number of chains
\(x<y<z\) through ranks \(a,b,c\).  Boolean inversion in
\cref{eq:flag-h-definition} and \cref{eq:shifted-flag-expansion} give
\[
  [t^3]\HH_{P_r,\rank}(q,t)
  =\sum_{0\leq a<b<c\leq 2r-3}\beta_{abc}\,q^{a+b+c},
  \qquad
  \beta_{abc}=n_{abc}-e_{ab}-e_{ac}-e_{bc}+f_a+f_b+f_c-1.
\]
Here \(\beta_{abc}\) is the reduced Euler characteristic of the order
complex of the subposet of ranks \(a\), \(b\), \(c\); it is nonnegative
by \cref{thm:r-labelling}, although this is not visible from the Euler
characteristic alone.
\end{remark}

\section{A coloured path formula in type \texorpdfstring{\(D\)}{D}}
\label{sec:type-D}

The uniform theorem proves positivity without a separate computation in
each root type.  In type \(D\), the numerator also has a direct path
formula closely related to \cref{thm:main}.  The doubled-column
description of the type-\(D\) root poset and its ordinary multichain
enumeration are given by Marsh and Williams
\cite[Section~4.3]{marshwilliams2022}.  We use coordinates adapted to
root height and give a signed EL-labelling that retains the \(q\)-weight.

\subsection{The doubled cell model}

Let \(r\geq4\) and put \(n=r-1\).  In the balanced cell poset
\[
  C_n=\{(i,c):1\leq i\leq n,\ i\leq c\leq2n-i\},
\]
replace each cell \((i,n)\) by two cells \((i,n,-)\) and \((i,n,+)\).
Within either copy of the middle column the original order is retained;
cells in different copies are incomparable.  All comparisons involving a
cell outside the middle column are inherited from \(C_n\).
Denote this poset by \(C_n^D\), with rank \(\rho(i,c)=c-i\), ignoring
the colour.  Adjoin a new minimum \(\bottom\) of rank zero and give
every cell the shifted rank \(c-i+1\).

\begin{lemma}
\label{lem:cell-model-D}
With simple roots
\[
  \alpha_i=e_i-e_{i+1}\quad(1\leq i\leq n),\qquad
  \alpha_{n+1}=e_n+e_{n+1},
\]
the following map is a rank-preserving isomorphism from
\(P_r^D=\PhiPlus(D_r)\), graded by height minus one, to \(C_n^D\):
\begin{align}
  e_i-e_j&\longmapsto(i,j-1) &&(i<j\leq n),\nonumber\\
  e_i+e_j&\longmapsto(i,2n-j+1) &&(i<j\leq n),\label{eq:root-cell-map-D}\\
  e_i-e_{n+1}&\longmapsto(i,n,-) &&(1\leq i\leq n),\nonumber\\
  e_i+e_{n+1}&\longmapsto(i,n,+) &&(1\leq i\leq n).\nonumber
\end{align}
\end{lemma}

\begin{proof}
The difference roots with \(j\leq n\) fill columns \(i,\ldots,n-1\)
of row \(i\); the sum roots with \(j\leq n\) fill columns
\(n+1,\ldots,2n-i\).  The two remaining roots fill the two middle
cells.  Thus the map is bijective.

For a cell in column \(c<n\), its root has coefficients one on
\(\alpha_i,\ldots,\alpha_c\), and zero elsewhere.  A middle cell has
coefficients one on \(\alpha_i,\ldots,\alpha_{n-1}\), with final pair
\[
  (\alpha_n,\alpha_{n+1}):\quad
  (1,0)\text{ for }-,\qquad (0,1)\text{ for }+.
\]
An empty interval of indices contributes no coefficients.
For \(c>n\), put \(j=2n-c+1\).  Its root has coefficients one on
\(\alpha_i,\ldots,\alpha_{j-1}\), two on
\(\alpha_j,\ldots,\alpha_{n-1}\), and final pair \((1,1)\).

Comparison in each of these three regions is exactly
\(i\geq i'\), \(c\leq c'\), as in the type-\(B\) cell model.
The final pairs force every comparison between different regions to
go from the smaller column to the larger one; support containment then
gives the row condition.  In the middle column the two final pairs are
incomparable, while within a fixed colour comparison is again by row.
These are precisely the defining comparisons of \(C_n^D\).
The sum of the coefficients in every case is \(c-i+1\), proving the
rank assertion.
\end{proof}

\subsection{Labels and coloured lifts}

Put \(K=2n+1\).  Retain the signed labels from
\cref{eq:explicit-labels}, with two exceptions: a right move entering
\((i,n,+)\) has label \(K\), and the edge from \(\bottom\) to the
extra atom \((n,n,+)\) also has label \(K\).  Thus left moves into row
\(i\) have label \(-i\); ordinary right moves into column \(c\) have
label \(c\); and the other atom edges have label \(i\).
Vertical moves within either middle copy and right moves leaving it have
their ordinary labels.

\begin{proposition}
\label{prop:type-D-el}
These labels form an EL-labelling of
\(\widehat C_n^D=C_n^D\sqcup\{\bottom\}\), with distinct labels along
every maximal chain.
\end{proposition}

\begin{proof}
A chain can enter the middle column only once and cannot change colour
there.  Its negative left labels are distinct and increasing, as are
its ordinary right labels.  The atom label precedes any ordinary right
label of the chain, and \(K\) occurs at most once, so all labels are
distinct.

Consider an interval with nonbottom endpoints.  Away from the middle
column, or within one copy, the unique increasing chain performs all
left moves before all right moves.  It is legal because left moves relax
the upper boundary.  The same rule holds when the interval starts in
one middle copy.  If it ends in the \(+\) copy, the entry label \(K\)
is its final label.  If it crosses the middle column completely, an
entry into the \(+\) copy would be followed by a smaller label, hence
would create a descent.  The unique increasing chain therefore crosses
the \(-\) copy, after all left moves.  These cases include intervals
ending in the \(-\) copy and all intervals on one side of the column.
In each case all left choices have smaller labels than right choices,
and a \(-\) entry is smaller than a \(+\) entry.  The increasing chain
is consequently lexicographically first.

Now let the lower endpoint be \(\bottom\) and the upper endpoint be
in row \(j\).  Any chain starting at an atom in a larger row must
eventually make a left move, producing a descent after a positive
label.  Thus an increasing chain starts in row \(j\) and uses only
right moves.  If the upper endpoint lies beyond the middle column,
it must use the \(-\) copy; if it lies in a middle copy, that copy is
prescribed.  The edge to \((n,n,+)\) gives a one-edge chain when that
is the endpoint, and otherwise forces a descent on the necessary first
left move.  The resulting increasing chain is unique and starts with
the least possible atom label; its subsequent choices are also least.
Singleton intervals require no comparisons.  This proves the EL property
on every interval.
\end{proof}

Every maximal chain of the augmented type-\(B_n\) cell poset passes
through the middle column.  Its move word \(w\in\mathcal B_n\) has
two lifts to \(\widehat C_n^D\), one for each middle colour.
To describe their weights, start the path at
\((\iota(w),\iota(w))\), where \(\iota(w)=1+\#L(w)\), and let
\(p(w)\) be the number of moves made when the path first reaches
column \(n\).  In particular, \(p(w)=0\) if it starts there.
Since the terminal column is \(2n-1>n\), the next move exists.  Put
\[
  \varepsilon(w)=
  \begin{cases}
    1,&w_{p(w)+1}=R,\\
    0,&w_{p(w)+1}=L,
  \end{cases}
  \qquad E(w)=\sum_{j\in D(w)}(j-1).
\]

\begin{theorem}[Type-\(D\) coloured-ballot formula]
\label{thm:type-D-path}
For every \(r\geq4\), with \(n=r-1\) and \(\rho=\height-1\),
\begin{equation}
\label{eq:type-D-path}
  \HH_{P_r^D,\rho}(q,t)
  =\sum_{w\in\mathcal B_n}
    q^{E(w)}t^{|D(w)|}
    \bigl(1+(q^{p(w)}t)^{\varepsilon(w)}\bigr).
\end{equation}
Here the fixed denominator is
\(\prod_{a=0}^{2r-4}(1-q^at)\).
\end{theorem}

The lift factor is \(2\) when \(\varepsilon(w)=0\), and
\(1+q^{p(w)}t\) when \(\varepsilon(w)=1\): the two colours either
have the same descent weight or differ by one additional descent.

\begin{proof}
Forgetting the colour projects each maximal chain to a chain in the
balanced type-\(B_n\) cell poset.  Conversely, the segment in column
\(n\) can be given either colour, and all other cells and moves are
unchanged.  This gives exactly two lifts, including chains starting
at either of the two atoms in row \(n\).

The \(-\) lift has the old descent set \(D(w)\).  For the \(+\) lift,
only the entry label changes, from its ordinary positive value to
\(K\).  The preceding comparison, if present, remains increasing,
because every preceding label is smaller than \(K\) and was already
smaller than the ordinary entry label.  If the next move is left, its
negative label produces the same descent as before.  If the next move
is right, its label \(n+1<K\) produces one additional descent.
All later comparisons are unchanged.  This additional descent is at
augmented rank \(p(w)+1\), so its contribution to the unaugmented
numerator is \(q^{p(w)}t\), by
\cref{prop:shifted-flag-expansion}.

When \(p(w)=0\), the starting row is \(n\).  A right move there is
not legal, so the first move is left and \(\varepsilon(w)=0\).
There is therefore no exceptional initial descent to add.
The signed EL-labelling and minimum-adjunction identity now give
\cref{eq:type-D-path} by summing the two lift weights.  The top cell
\((1,2n-1)\) has height \(2n-1\), so the original maximum rank is
\(2n-2=2r-4\), as stated.
\end{proof}

\begin{example}[Rank four]
\label{ex:type-D-four}
For \(D_4\), \(n=3\), and \(\mathcal B_3\) consists of the six words
\(RRRR\), \(LRRR\), \(RLRR\), \(RRLR\), \(LLRR\), \(LRLR\).  Their
triples \((D(w),p(w),\varepsilon(w))\) are
\((\varnothing,2,1)\), \((\{1\},2,1)\), \((\{2\},1,0)\),
\((\{3\},1,1)\), \((\{1\},0,0)\), and \((\{1,3\},0,0)\),
so \cref{eq:type-D-path} gives
\[
  \HH_{P_4^D,\rho}(q,t)
  =1+(3+2q+2q^2)\,t+(3q^2+q^3)\,t^2 .
\]
This agrees with the flag numbers of \(L_{D_4}\) inserted in
\cref{prop:shifted-flag-expansion}: the twelve positive roots have
\(4,3,3,1,1\) elements of heights \(1,\ldots,5\), so
\(\beta(\{1\})=3\) and \(\beta(\{2\})=\beta(\{3\})=2\); the nine
comparable pairs between heights one and three give
\(\beta(\{1,3\})=3\), the six between heights two and three give
\(\beta(\{2,3\})=1\), and every other nonempty flag \(h\)-number vanishes.  In
particular \(\HH_{P_4^D,\rho}(1,1)=12\) is the number of maximal chains
of \(L_{D_4}\).
\end{example}

At \(q=1\), \cref{thm:type-D-path} gives the numerator of the ordinary
weak-chain series as the descent distribution of the coloured maximal
chains; for \(D_4\) it is \(1+7t+4t^2\).  Unlike type \(B\), this is not
a Narayana polynomial of the root system (the type-\(D_4\) Narayana
polynomial has degree four), and we do not know a Coxeter-theoretic
interpretation of these coefficients.

\section{A two-parameter trapezoidal extension}
\label{sec:trapezoid-family}

We now widen the balanced boundary used for the type-\(B/C\) root poset.
The signed labels and minimum-adjunction identity continue to apply;
the paths acquire an additional endpoint parameter.
Let \(r\) and \(s\) be integers with \(s\geq 2r\geq2\), and define
\[
  T_{r,s}=\{(i,c):1\leq i\leq r,\ i\leq c\leq s-i\},
\]
ordered by
\[
  (i,c)\leq(j,d)\quad\Longleftrightarrow\quad i\geq j\ \text{and}\ c\leq d,
\]
and graded by \(\rho(i,c)=c-i\).  Adjoin a minimum \(\bottom\), and give
the new cells rank \(\widehat\rho(i,c)=c-i+1\).  Put \(m=s-2\), and let
\(\mathcal B_{r,s}\) consist of the words \(w\in\{L,R\}^{m}\) such that
\[
  \text{the reversal of \(w\) is ballot,}
  \qquad \#R(w)-\#L(w)\geq s-2r.
\]

For \(w=w_1\cdots w_m\), retain the run-start set
\[
  D(w)=\{j:w_j=L\text{ and }(j=1\text{ or }w_{j-1}=R)\}.
\]

\begin{theorem}[Trapezoidal-cell \(q\)-Zeta formula]
\label{thm:trapezoid-family}
For every \(s\geq2r\geq2\), the \(q\)-Zeta numerator of \(T_{r,s}\), with rank
\(\rho(i,c)=c-i\), is
\[
  \HH_{T_{r,s},\rho}(q,t)
  =\sum_{w\in\mathcal B_{r,s}}
    q^{\sum_{j\in D(w)}(j-1)}t^{|D(w)|}.
\]
In particular, it belongs to \(\NN[q,t]\).  For \(s=2r\), one has
\(\mathcal B_{r,2r}=\mathcal B_r\) and \(T_{r,2r}=C_r\).
\end{theorem}

\begin{proof}[Proof of \cref{thm:trapezoid-family}]
Adjoin a minimum to \(T_{r,s}\).  For nonbottom cells, the join and meet
are
\[
  (i,c)\vee(j,d)=(\min(i,j),\max(c,d)),
\]
and
\[
  (i,c)\wedge(j,d)=
  \begin{cases}
    (\max(i,j),\min(c,d)),&\max(i,j)\leq\min(c,d),\\
    \bottom,&\max(i,j)>\min(c,d).
  \end{cases}
\]
The inequalities \(c\leq s-i\) and \(d\leq s-j\) show that the proposed
join is a cell, while the same argument for the proposed meet applies when
it is nonbottom.  The rank is \(\widehat\rho=c-i+1\).  If the meet is
bottom, assume \(i\geq j\); then \(d<i\leq c\), and
\[
  \widehat\rho(i,c)+\widehat\rho(j,d)-\widehat\rho(j,c)=d-i+1\leq0.
\]
Thus the augmented poset is a lower-semimodular lattice.

Use the signed labelling of \cref{eq:explicit-labels}.
It is an EL-labelling by the general Ferrers-cell construction in the
proof of \cref{thm:ferrers-family}, applied to \(b_i=s-i\).
That construction uses only the monotonicity of the boundary and is
independent of the path enumeration below.

A maximal chain starts at \((i,i)\), uses \(i-1\) left moves and \(s-i-1\)
right moves, and ends at \((1,s-1)\).  If \(L_k,R_k\) count moves in a
prefix, the intermediate cell is \((i-L_k,i+R_k)\), and the upper boundary
condition is
\[
  R_k-L_k\leq s-2i.
\]
At the end equality holds.  Consequently the reversed word is ballot, with
final balance \(b=s-2i\geq s-2r\).  Conversely, if \(w\in\mathcal B_{r,s}\)
has final balance \(b\), then \(b\equiv s-2\equiv s\pmod 2\), so
\(i=(s-b)/2\) is an integer in \([1,r]\), and
the same inequalities keep the path inside \(T_{r,s}\).  This gives a
bijection between maximal chains and \(\mathcal B_{r,s}\).

The label comparison is identical to the root-poset case: descents occur at
rank \(1\) when the first move is \(L\), and thereafter exactly at positions
where \(R\) is followed by \(L\).  Thus the descent set is \(D(w)\).  Applying
the flag expansion to the EL-labelled augmented lattice and then the exact
minimum-adjunction shift from \cref{lem:minimum-shift} gives the displayed
formula.
\end{proof}

\subsection{\texorpdfstring{A closed endpoint formula for all \(t\)-slices}
  {A closed endpoint formula for all t-slices}}
\label{sec:trapezoid-endpoint-formula}

The ballot-word formula can be converted into a finite Gaussian-binomial
expression for the complete numerator, rather than only for its two extreme
slices.  For integers \(0\leq k\leq N\), write
\[
  \qbinom{N}{k}
  =\prod_{u=1}^{k}\frac{1-q^{N-k+u}}{1-q^u},
\]
with the empty product equal to one.  For integer pairs \((N,k)\) outside
\(0\leq k\leq N\), we set both \(\qbinom{N}{k}\) and
\(\binom{N}{k}\) equal to zero.
We use the following standard one-boundary enumerator.  If
\(a,b,d\geq0\) and \(b\leq a+d\), let
\begin{equation}
\label{eq:one-boundary-enumerator}
 \Phi_{a,b}^{d}(x,q)
 =\sum_{\nu\geq0}x^{\nu}q^{\nu^2}
 \left(
 \qbinom{a}{\nu}\qbinom{b}{\nu}
 -q^{d}\qbinom{b-d+1}{\nu+1}
          \qbinom{a+d-1}{\nu-1}
 \right).
\end{equation}
By the one-boundary descent--major-index formula
of Krattenthaler and Mohanty \cite[Corollary~2]{krattenthaler1993},
\(\Phi_{a,b}^{d}(x,q)\) is the generating function
\[
 \sum_{\pi}x^{\operatorname{des}(\pi)}q^{\operatorname{maj}(\pi)},
\]
where \(\pi\) runs over words with \(a\) letters \(L\) and \(b\) letters
\(R\) whose associated path satisfies \(R_k-L_k\leq d\) at every prefix,
\(\operatorname{des}(\pi)=|\{i:\pi_i=R,\pi_{i+1}=L\}|\), and
\(\operatorname{maj}(\pi)\) is the sum of the descent positions.

For the slack \(d=b-a\), which is the case needed in
\cref{thm:trapezoid-endpoint}, these polynomials can also be read off from
Chu's upper-corner polynomials.  Write \(B^+_{m,n}(q,x)\) for Chu's
enumerator of subdiagonal paths with \(m\) horizontal and \(n\) vertical
steps, where \(x\) marks upper corners and \(q\) marks the sum of their
coordinate sums.  For \(0\leq a\leq b\),
\[
  \Phi_{a,b}^{\,b-a}(x,q)=B^+_{b,a}(q^{-1},xq^{a+b}).
\]
Reversing the move word and taking \(R\) as horizontal and \(L\) as
vertical sends an \(RL\) descent at position \(i\) to an upper corner
with coordinate sum \(a+b-i\).  This gives Chu's dual upper-corner
enumerator \cite[Equation~(10a) and Corollary~6, Equation~(17a)]{chu2017qballot};
empty and one-direction paths contribute \(1\).

The proof below returns to the Krattenthaler--Mohanty coordinate
convention, with \(L\) horizontal and \(R\) vertical.

\begin{theorem}[Gaussian-binomial endpoint decomposition]
\label{thm:trapezoid-endpoint}
Put \(m=s-2\).  For \(1\leq a\leq r-1\), set
\[
 b=m-a,\qquad d=b-a=m-2a,
\]
and define
\begin{equation}
\label{eq:endpoint-contribution}
 \mathcal G_{a,b}(q,t)
 =\Phi_{a,b}^{d}(t,q)
 +(t-1)\Phi_{a-1,b}^{d+1}(qt,q).
\end{equation}
Then
\begin{equation}
\label{eq:trapezoid-endpoint-sum}
 \HH_{T_{r,s},\rho}(q,t)
 =1+\sum_{a=1}^{r-1}\mathcal G_{a,m-a}(q,t).
\end{equation}
Thus every \(t\)-slice of the trapezoidal \(q\)-Zeta numerator has an explicit
finite Gaussian-binomial expression.
\end{theorem}

\begin{proof}[Proof of \cref{thm:trapezoid-endpoint}]
Fix a word \(w\in\mathcal B_{r,s}\) with \(a\) letters \(L\) and \(b=m-a\)
letters \(R\).  Its reversed-ballot condition is equivalent to
\[
 R_k-L_k\leq b-a=d
\]
for every prefix of \(w\), and \(a\leq r-1\).  Let
\[
 S(w)=\{i:w_i=R,\ w_{i+1}=L\}.
\]
The run-start set in the \(q\)-Zeta formula satisfies
\[
 D(w)=\bigl(\{1\}\text{ if }w_1=L\text{, otherwise }\varnothing\bigr)
       \cup\{i+1:i\in S(w)\}.
\]
Consequently
\[
 \sum_{j\in D(w)}(j-1)=\operatorname{maj}(w),
 \qquad |D(w)|=|S(w)|+\mathbf 1_{\{w_1=L\}}.
\]
The first term in \cref{eq:endpoint-contribution} counts all such paths with
the ordinary descent weight.  For the paths beginning with \(L\), write
\(w=Lu\).  Deleting that initial horizontal step gives a path from the
origin to \((a-1,b)\) under the boundary \(R_k-L_k\leq d+1\).  Moreover,
\[
 \operatorname{des}(w)=\operatorname{des}(u),
 \qquad
 \operatorname{maj}(w)=\operatorname{maj}(u)+\operatorname{des}(u).
\]
Their ordinary descent generating function is therefore
\(\Phi_{a-1,b}^{d+1}(qt,q)\).  This is the initial-horizontal-step
identity of \cite[Equation~(3.5)]{krattenthaler1993}, in our coordinates.
The extra initial run multiplies it by
\(t\) rather than \(1\).  Replacing the ordinary contribution by this extra
factor gives \cref{eq:endpoint-contribution}.  The all-\(R\) word contributes
the initial \(1\), and summing over \(a\) proves
\cref{eq:trapezoid-endpoint-sum}.
\end{proof}

\begin{theorem}[Intervals not containing the minimum]
\label{thm:trapezoid-interval-endpoint}
Let \(x=(i,c)\leq y=(j,d)\) be nonbottom cells of \(T_{r,s}\), and put
\[
 \ell=i-j,\qquad u=d-c,\qquad \delta=s-i-c.
\]
For the induced rank \(\rho_{x,y}(z)=\rho(z)-\rho(x)\) on
the interval \([x,y]\), one has
\begin{equation}
\label{eq:trapezoid-interval-gaussian}
 \HH_{[x,y]}(q,t)=\Phi_{\ell,u}^{\delta}(t,q).
\end{equation}
In particular,
\begin{equation}
\label{eq:trapezoid-interval-total}
 \HH_{[x,y]}(1,1)
 =\binom{\ell+u}{\ell}
  -\binom{\ell+u}{\ell+\delta+1}.
\end{equation}
Thus every nonbottom interval has a complete \(q\)-binomial \(q\)-Zeta numerator,
depending only on its displacement \((\ell,u)\) and its boundary slack
\(\delta\).
\end{theorem}

\begin{proof}[Proof of \cref{thm:trapezoid-interval-endpoint}]
If \(x=y\), the weak-chain series is \((1-t)^{-1}\), so the numerator
is \(1=\Phi_{0,0}^{\delta}(t,q)\).  Assume henceforth that \(x<y\).
A maximal chain in \([x,y]\) is a word with \(\ell\) letters \(L\) and
\(u\) letters \(R\).  After a prefix with counts \(L_k,R_k\), its cell is
\((i-L_k,c+R_k)\).  The sole nonautomatic cell condition is
\[
 c+R_k\leq s-(i-L_k),
 \qquad\text{equivalently}\qquad
 R_k-L_k\leq\delta.
\]
The local descent set is
\(\{p:w_p=R,\ w_{p+1}=L\}\), so its \(q\)-Zeta weight is the ordinary
descent--major-index weight.  \Cref{eq:one-boundary-enumerator}
therefore gives \cref{eq:trapezoid-interval-gaussian}.  Evaluating
\(\Phi_{\ell,u}^{\delta}(1,q)\) by the one-boundary major-index specialization
\cite[Corollary~4]{krattenthaler1993} yields
\[
 \Phi_{\ell,u}^{\delta}(1,q)
 =\qbinom{\ell+u}{\ell}
  -q^{\delta+1}
    \qbinom{\ell+u}{\ell+\delta+1},
\]
and \(q=1\) gives \cref{eq:trapezoid-interval-total}.
\end{proof}

\begin{definition}[Endpoint contribution with a prescribed initial atom]
\label{def:trapezoid-principal-contribution}
For \(a,b,d\geq0\) with \(b\leq a+d\), define
\[
 \Psi_{a,b}^{d}(q,t)=
 \begin{cases}
  1,&a=0,\\[2pt]
  \Phi_{a,b}^{d}(t,q)
  +(t-1)\Phi_{a-1,b}^{d+1}(qt,q),&a\geq1.
 \end{cases}
\]
\end{definition}

\begin{theorem}[Principal intervals]
\label{thm:trapezoid-principal-endpoint}
Let \(y=(j,d)\in T_{r,s}\), put \(K=\min\{r,d\}\), and use the induced rank
\(\widehat\rho(z)-\widehat\rho(\bottom)=\widehat\rho(z)\), with
\(\widehat\rho(\bottom)=0\), on the principal interval
\([\bottom,y]\).  Then
\begin{equation}
\label{eq:trapezoid-principal-endpoint}
\HH_{[\bottom,y]}(q,t)
 =\sum_{k=j}^{K}
   \Psi_{k-j,\,d-k}^{\,s-2k}(q,qt).
\end{equation}
Consequently, every interval of the augmented trapezoidal lattice has a
closed Gaussian-binomial \(q\)-Zeta numerator: nonbottom intervals are given by
\cref{thm:trapezoid-interval-endpoint}, and bottom intervals by
\cref{eq:trapezoid-principal-endpoint}; the remaining singleton
\([\bottom,\bottom]\) has numerator \(1\).
\end{theorem}

\begin{proof}[Proof of \cref{thm:trapezoid-principal-endpoint}]
A maximal chain in \([\bottom,y]\) chooses an initial atom \((k,k)\) with
\(j\leq k\leq K\), followed by \(k-j\) left moves and \(d-k\) right moves.
For a prefix with \(L_p,R_p\) moves, the cell is
\((k-L_p,k+R_p)\); the only nonautomatic condition is
\[
 R_p-L_p\leq s-2k.
\]
The edge from \(\bottom\) to \((k,k)\) has label \(k\).  It is followed by a
descent exactly when the first move is \(L\), while all later descents are
the ordinary \(R\)-to-\(L\) descents.  Hence the descent ranks of the chain
inside the bounded interval \([\bottom,y]\) form exactly the run-start set
\(D(w)\) of its move word \(w\), and the flag expansion
\cref{thm:flag-expansion} weights the chain by
\[
 q^{\sum_{p\in D(w)}p}\,t^{|D(w)|}
 =q^{\operatorname{maj}(w)}(qt)^{|D(w)|},
\]
using \(\sum_{p\in D(w)}(p-1)=\operatorname{maj}(w)\) as in the proof of
\cref{thm:trapezoid-endpoint}.  For fixed \(k\), the deletion argument of
\cref{thm:trapezoid-endpoint} evaluates the generating function of
\(q^{\operatorname{maj}(w)}\tau^{|D(w)|}\) over these words as
\(\Psi_{k-j,\,d-k}^{\,s-2k}(q,\tau)\); when \(k=j\) the word is all \(R\)
and contributes \(1\).  Substituting \(\tau=qt\) gives
\cref{eq:trapezoid-principal-endpoint}.  The endpoint inequality
\(d+j\leq s\) guarantees \(d-k\leq(k-j)+(s-2k)\), so the one-boundary
formula applies to every summand.
\end{proof}

\begin{corollary}[Intervals at \(q=1\)]
\label{cor:trapezoid-interval-qone-profile}
With the notation of \cref{thm:trapezoid-interval-endpoint}, the complete
univariate specialization is
\begin{equation}
\label{eq:trapezoid-interval-qone-profile}
\HH_{[x,y]}(1,t)
=\sum_{\nu\geq0} t^\nu
\left(
\binom{\ell}{\nu}\binom{u}{\nu}
-\binom{u-\delta+1}{\nu+1}
 \binom{\ell+\delta-1}{\nu-1}
\right).
\end{equation}
When the upper boundary is inactive (\(\delta\geq u\)), this reduces to
\(\sum_{\nu\geq0}\binom{\ell}{\nu}\binom{u}{\nu}t^\nu\).
\end{corollary}

\begin{proof}
Set \(q=1\) in \cref{eq:trapezoid-interval-gaussian} and use the definition
of \(\Phi_{\ell,u}^{\delta}\).  If \(\delta\geq u\), the second binomial
product vanishes for every \(\nu\), giving the stated rectangular
specialization.
\end{proof}

\begin{corollary}[The specialization \(t=1\)]
\label{cor:trapezoid-q-total}
With \(m=s-2\), one has
\begin{equation}
\label{eq:trapezoid-q-total}
 \HH_{T_{r,s},\rho}(q,1)
 =1+\sum_{a=1}^{r-1}
 \left(
 \qbinom{m}{a}
 -q^{m-2a+1}\qbinom{m}{a-1}
 \right).
\end{equation}
At \(q=1\), this reduces to
\(\HH_{T_{r,s},\rho}(1,1)=\binom{s-2}{r-1}\).
\end{corollary}

\begin{proof}
Set \(t=1\) in \cref{thm:trapezoid-endpoint}; the initial-left-run
correction vanishes.  The one-boundary major-index specialization
\cite[Corollary~4]{krattenthaler1993} gives
\[
 \Phi_{a,b}^{d}(1,q)
 =\qbinom{a+b}{a}
  -q^{d+1}\qbinom{a+b}{a+d+1}.
\]
Here \(a+b=m\), \(d=m-2a\), and
\(\qbinom{m}{a+d+1}=\qbinom{m}{m-a+1}=\qbinom{m}{a-1}\).
This proves \cref{eq:trapezoid-q-total}.  At \(q=1\), the summand is
\(\binom{m}{a}-\binom{m}{a-1}\), so the sum telescopes.
\end{proof}

\subsection{Profiles and zeros at \texorpdfstring{\(q=1\)}{q=1}}

At \(q=1\) the trapezoidal numerators are the binomial-product
polynomials \(\sum_k\binom{A}{k}\binom{B}{k}t^k\), that is, Jacobi
polynomials after a fractional-linear change of variables.  The
statements about zeros in this subsection are therefore consequences of
classical facts on orthogonal polynomials; we include them because they
describe the whole trapezoidal family uniformly.

\begin{corollary}[The trapezoidal profile at \(q=1\)]
\label{cor:trapezoid-qone-profile}
For \(s\geq 2r\geq2\),
\begin{equation}
\label{eq:trapezoid-qone-profile}
\HH_{T_{r,s},\rho}(1,t)
 =\sum_{k=0}^{r-1}
   \binom{r-1}{k}\binom{s-r-1}{k}t^k.
\end{equation}
In particular, this profile is real-rooted with simple negative zeros.  More
precisely, writing \(A=r-1\) and \(B=s-r-1\), one has
\[
\HH_{T_{r,s},\rho}(1,t)
 =(1-t)^A
 P_A^{(0,B-A)}\!\left(\frac{1+t}{1-t}\right),
\]
where \(P_A^{(\alpha,\beta)}\) is the Jacobi polynomial.
\end{corollary}

\begin{proof}
Recall that \(m=s-2\).  For \(1\leq a\leq r-1\), let \(b=m-a\).  From
\cref{thm:trapezoid-endpoint} and the \(q=1\) specialization of
\cref{eq:one-boundary-enumerator}, the coefficient of \(t^k\), for \(k\geq1\),
in the \(a\)-th endpoint contribution is
\[
\binom{a-1}{k-1}\binom{m-a+1}{k}
-\binom{a}{k}\binom{m-a}{k-1}.
\]
The constant term is supplied only by the all-\(R\) word.  Two
applications of Pascal's identity show that this summand equals
\(\binom{a}{k}\binom{m-a}{k}-\binom{a-1}{k}\binom{m-a+1}{k}\).  Summing
over \(1\leq a\leq A\), with \(A=r-1\), the sum telescopes to
\(\binom{A}{k}\binom{m-A}{k}\), which gives
\cref{eq:trapezoid-qone-profile} since \(m-A=s-r-1\).

Finally, the case \(\alpha=0\), \(\beta=B-A\), \(n=A\) of the explicit
formula \cite[Equation~(4.3.2)]{szego1939} for Jacobi polynomials, with
\(x=(1+t)/(1-t)\), gives
\[
\sum_{k=0}^{A}\binom{A}{k}\binom{B}{k}t^k
=(1-t)^A P_A^{(0,B-A)}
  \!\left(\frac{1+t}{1-t}\right).
\]
Since \(B\geq A\), the Jacobi parameters are nonnegative, so the \(A\)
zeros of \(P_A^{(0,B-A)}\) are simple and lie in \((-1,1)\)
\cite[Theorem~3.3.1]{szego1939}.  The change of variables
\(t=(x-1)/(x+1)\) sends them to \(A\) simple negative zeros.
\end{proof}

The ordinary trapezoid multichain count goes back to Proctor
\cite{proctor1983} and Stembridge (Corollary~2.2 of
\cite{stembridge1986}), as recalled by
Marsh and Williams \cite[Section~4.2, Equation~(4)]{marshwilliams2022}:
the trapezoid \(T_{r,s}\) and the rectangle \([r]\times[s-r]\) have the
same number of weak chains of length \(k\) for every \(k\), hence the same zeta
polynomial, and \cref{eq:trapezoid-qone-profile} is their common
numerator.  For Stembridge's parameters \(m=r\) and \(n=s-r\), the map
\((i,c)\mapsto(i,s-c)\) identifies our trapezoid with the dual of
his product-order trapezoid.
In his coordinates, the rectangle has weights \(w_R(i,j)=x_i x_j\),
whereas the trapezoid has the piecewise weights
\[
  w_T(i,j)=
  \begin{cases}
    x_i x_{j+1},&j<n,\\
    x_i x_{m+n-j},&j\geq n.
  \end{cases}
\]
Theorem~2.1 of \cite[p.~379]{stembridge1986} equates the multichain
totals for these weights. Under the coordinate map above, our cell
weight \(q^{c-i}\) becomes \(q^{s-i-j}\), rather than one of these
native weights. For our rank weights, the two
numerators need not agree.  Already the rank generating functions differ:
\(T_{3,6}\) has \(3+2q+2q^2+q^3+q^4\), whereas \([3]\times[3]\) has
\(1+2q+3q^2+2q^3+q^4\); accordingly
\(\HH_{T_{3,6},\rho}=1+(2+q+q^2)t+q^2t^2\), whereas the bounded
distributive lattice \([3]\times[3]\), graded by \((i-1)+(j-1)\), has
numerator \(1+(q+2q^2+q^3)t+q^4t^2\) by \cref{thm:flag-expansion}.  The endpoint decomposition above evaluates
the descent weights of the Ferrers-cell labelling, which retain this
rank information before the specialization \(q=1\).
This comparison distinguishes the stated weightings; it does not exclude
an indirect derivation from a further specialization or transformation of
Stembridge's multivariate identity.

The associated rational kernel, fixed-offset generating functions,
interlacing, and limiting zero distribution are collected in
\ref{app:classical-trapezoids}.

\subsection{Sharp coefficient bounds}

\begin{corollary}[Degrees of the \(t\)-slices]
\label{cor:trapezoid-slice-degree}
For \(0\leq k\leq r-1\),
\[
  [t^k]\HH_{T_{r,s},\rho}(q,t)
  \in q^{k(k-1)}\NN[q],
  \qquad
  \deg_q [t^k]\HH_{T_{r,s},\rho}(q,t)=k(s-k-3).
\]
The coefficient of \(q^{k(s-k-3)}\) is \(1\).
\end{corollary}

\begin{proof}
For a word \(w\) with \(D(w)=\{j_1<\cdots<j_k\}\), consecutive run starts
are separated by a right letter, so \(j_a\geq 2a-1\). Hence
\[
  \sum_{a=1}^k(j_a-1)\geq k(k-1).
\]
Let \(v\) be the reversal of \(w\), and let
\(p_1<\cdots<p_k\) be the reversed positions corresponding to the elements
of \(D(w)\). Each \(p_a\) is an \(L\)-position of the ballot word \(v\).
Among its first \(p_a\) letters there are at least \(a\) such positions, and
the ballot condition forces at least \(a\) \(R\)'s there; thus \(p_a\geq2a\).
Writing \(m=s-2\),
\[
  \sum_{a=1}^k(j_a-1)=\sum_{a=1}^k(m-p_a)
  \leq \sum_{a=1}^k(m-2a)=k(s-k-3).
\]
Equality forces \(p_a=2a\) for every \(a\). The ballot condition then forces
\(v\) to begin with \((RL)^k\); since there are exactly \(k\) marked
positions, all remaining letters are \(R\). Therefore
\(v=(RL)^kR^{s-2-2k}\), which is admissible precisely when \(k\leq r-1\).
This unique word attains the upper bound.
\end{proof}

\begin{corollary}[Specialization of the trapezoidal family]
\label{cor:trapezoid-specialization}
For \(s\geq2r\geq2\),
\[
  \HH_{T_{r,s},\rho}(1,1)=\binom{s-2}{r-1},
  \qquad
  \deg_t\HH_{T_{r,s},\rho}=r-1,
\]
and
\[
  [t^{r-1}]\HH_{T_{r,s},\rho}(1,t)
  =\binom{s-r-1}{r-1}.
\]
\end{corollary}

\begin{proof}
The first identity is the case \(q=1\) of \cref{cor:trapezoid-q-total}.
The third is the case \(k=r-1\) of \cref{eq:trapezoid-qone-profile}.
Since \(s\geq2r\), the binomial coefficient \(\binom{s-r-1}{r-1}\) is
positive, and \cref{eq:trapezoid-qone-profile} has no terms of higher
degree.  As all coefficients of \(\HH_{T_{r,s},\rho}\) lie in \(\NN[q]\),
the \(t\)-degree of \(\HH_{T_{r,s},\rho}(q,t)\) equals that of its
specialization at \(q=1\), namely \(r-1\).
\end{proof}

\section{Arbitrary Ferrers boundaries}
\label{sec:ferrers-family}

The trapezoidal boundary is only one instance of a monotone Ferrers
boundary. Let \(r\geq1\) and let
\[
  b_1\geq b_2\geq\cdots\geq b_r\geq r
\]
be integers; equivalently, \(\mathbf b\) is nonincreasing and \(b_i\geq i\)
for every \(i\). Define
\[
  F_{\mathbf b}=\{(i,c):1\leq i\leq r,\ i\leq c\leq b_i\},
\]
with the same order
\[
  (i,c)\leq(j,d)\quad\Longleftrightarrow\quad
  i\geq j\ \text{and}\ c\leq d,
\]
and rank \(\rho(i,c)=c-i\). Put \(M=b_1-1\). For a word
\(w\in\{L,R\}^{M}\), write \(\ell(w)\) for its number of \(L\)'s and set
\(\iota(w)=\ell(w)+1\). Let \(\mathcal W_{\mathbf b}\) consist of the words
with \(\ell(w)\leq r-1\) such that, for every prefix with \(L_k\) letters
\(L\) and \(R_k\) letters \(R\),
\[
  \iota(w)+R_k\leq b_{\iota(w)-L_k}.
\]

\begin{theorem}[Ferrers-cell \(q\)-Zeta formula]
\label{thm:ferrers-family}
For every boundary sequence \(\mathbf b\) as above,
\[
  \HH_{F_{\mathbf b},\rho}(q,t)
  =\sum_{w\in\mathcal W_{\mathbf b}}
    q^{\sum_{j\in D(w)}(j-1)}t^{|D(w)|}.
\]
In particular, \(\HH_{F_{\mathbf b},\rho}(q,t)\in\NN[q,t]\). The
trapezoidal family \(T_{r,s}\) is obtained by \(b_i=s-i\).
\end{theorem}

\begin{proof}[Proof of \cref{thm:ferrers-family}]
Adjoin a minimum to \(F_{\mathbf b}\). The same coordinate formulas give
\[
  (i,c)\vee(j,d)=(\min(i,j),\max(c,d))
\]
and
\[
  (i,c)\wedge(j,d)=
  \begin{cases}
    (\max(i,j),\min(c,d)),&\max(i,j)\leq\min(c,d),\\
    \bottom,&\text{otherwise}.
  \end{cases}
\]
The monotonicity of \(\mathbf b\) makes the displayed join and every
nonbottom meet a cell. Label the covers by
\[
  \lambda(\bottom,(i,i))=i,\qquad
  \lambda((i,c),(i-1,c))=-(i-1),\qquad
  \lambda((i,c),(i,c+1))=c+1.
\]
The covers are exactly the displayed left and right moves, and each increases
\(\rho(i,c)=c-i\) by one; thus \(\rho\) is a rank function on \(F_{\mathbf b}\).
On an interval not starting at the minimum, all left labels are negative and
increase as the first coordinate decreases, while all right labels are
positive and increase as the second coordinate increases. Hence the unique
increasing chain takes all available left moves before all right moves and is
lexicographically first. On an interval from the minimum to \((j,d)\), the
increasing chain starts with the least-labelled available atom \((j,j)\),
followed only by right moves.
Thus \(\lambda\) is an EL-labelling.

A maximal chain starts at \((i,i)\), makes \(i-1\) left moves and
\(b_1-i\) right moves, and ends at \((1,b_1)\). Its move word has length
\(M\), and the cell after a prefix is \((i-L_k,i+R_k)\). Therefore the path
stays in \(F_{\mathbf b}\) exactly when the defining inequalities for
\(\mathcal W_{\mathbf b}\) hold, with \(i=\iota(w)\). This gives a bijection
between maximal chains and \(\mathcal W_{\mathbf b}\). A first left move
follows a positive atom label and creates a descent; thereafter the only
descents occur when a right move is followed by a left move, since labels
increase within each run of like moves. Thus the descent set is \(D(w)\). The EL
flag expansion and the minimum-adjunction shift now give the displayed
formula.
\end{proof}

\subsection{Descent refinements and transfer recursion}

\begin{proposition}[Multivariate descent-set refinement]
\label{prop:ferrers-multivariate-refinement}
Let $M=b_1-1$ and let $\mathbf u=(u_1,\ldots,u_M)$ be commuting
indeterminates.  The multivariate descent enumerator
\[
  \mathcal D_{\mathbf b}(\mathbf u)
  =\sum_{w\in\mathcal W_{\mathbf b}}\prod_{j\in D(w)}u_j
\]
has nonnegative coefficients, and the \(q\)-Zeta numerator is its principal
specialization
\begin{equation}
\label{eq:ferrers-multivariate-specialization}
  \HH_{F_{\mathbf b},\rho}(q,t)
  =\mathcal D_{\mathbf b}(t,tq,tq^2,\ldots,tq^{M-1}).
\end{equation}
More generally, for a nonbottom interval $I=[x,y]$ of length $h\geq1$, use the
legal move words \(\mathcal W_{\mathbf b}(x,y)\), local \(RL\)-descent
positions \(D_{x,y}\), and rank normalized at \(x\), as defined explicitly
in \cref{prop:ferrers-interval-word}.  The local
descent enumerator
\[
  \mathcal D_I(\mathbf u)=
  \sum_{w\in\mathcal W_{\mathbf b}(x,y)}
       \prod_{p\in D_{x,y}(w)}u_p
\]
specializes to
\[
  \HH_{I,\rho_{x,y}}(q,t)=
  \mathcal D_I(tq,tq^2,\ldots,tq^{h-1}).
\]
By \cref{thm:r-labelling}, the coefficient of \(\prod_{j\in S}u_j\) in
\(\mathcal D_{\mathbf b}\) is the flag \(h\)-number \(\beta(S)\) of the
augmented poset; thus \(\mathcal D_{\mathbf b}\) records the whole flag
\(h\)-vector, and similarly for \(\mathcal D_I\).
\end{proposition}

\begin{proof}[Proof of \cref{prop:ferrers-multivariate-refinement}]
The first assertion is immediate from the definition: each admissible word
contributes one monomial with coefficient one.  In the global formula of
\cref{thm:ferrers-family}, substituting $u_j=tq^{j-1}$ gives
\[
  \prod_{j\in D(w)}u_j
  =t^{|D(w)|}q^{\sum_{j\in D(w)}(j-1)},
\]
which proves \cref{eq:ferrers-multivariate-specialization}.  For an interval,
the same substitution with $u_p=tq^p$ and the local word formula of
\cref{prop:ferrers-interval-word} gives the final identity.
\end{proof}

\begin{proposition}[Transfer-matrix recursion]
\label{prop:ferrers-transfer-matrix}
Let \(M=b_1-1\), and let \(\sigma\in\{\varnothing,L,R\}\) record the
previous move.  Define polynomials \(G_{j}(i,c;\sigma)\) recursively by
\(G_{0}(i,i;\varnothing)=1\) for \(1\leq i\leq r\), with all other
initial values zero.  For each \(1\leq j\leq M\), initialize all entries
of \(G_j\) to zero and add the following contributions from every state
\((i,c;\sigma)\) in layer \(j-1\):
\[
\begin{aligned}
G_j(i-1,c;L)&\mathrel{+}=G_{j-1}(i,c;\sigma)\cdot u_j^{\mathbf 1_{\{\sigma\in\{\varnothing,R\}\}}},
&& (i-1,c)\in F_{\mathbf b},\\
G_j(i,c+1;R)&\mathrel{+}=G_{j-1}(i,c;\sigma),
&& (i,c+1)\in F_{\mathbf b},
\end{aligned}
\]
for each \(\sigma\).  Then
\[
  \mathcal D_{\mathbf b}(\mathbf u)
  =\sum_{\sigma\in\{\varnothing,L,R\}}G_M(1,b_1;\sigma).
\]
In particular, the complete multivariate enumerator, and hence the \(q\)-Zeta
numerator by \cref{eq:ferrers-multivariate-specialization}, is computable with
\(O(M\lvert F_{\mathbf b}\rvert)\) polynomial-ring transitions.  (Since a
path reaches \((i,c)\) after exactly \(c-i\) moves, only the entries with
\(j=c-i\) are nonzero, so in fact \(O(\lvert F_{\mathbf b}\rvert)\)
transitions suffice.)
\end{proposition}

\begin{proof}[Proof of \cref{prop:ferrers-transfer-matrix}]
After \(j\) transitions, \(G_j(i,c;\sigma)\) is the sum of the monomials
attached to all legal length-\(j\) paths ending at \((i,c)\) whose last move is
\(\sigma\).  A left transition creates a descent exactly when it is the first
move or follows a right move, which is precisely the exponent in the first
line; right transitions never create a descent.  Induction on \(j\) proves the
state interpretation.  Every maximal path has length \(M\) and ends at
\((1,b_1)\), so summing the three terminal states gives the claimed
enumerator.  There are \(O(M\lvert F_{\mathbf b}\rvert)\) states and a bounded
number of transitions per state.
\end{proof}

\subsection{Characteristic polynomials}

\begin{corollary}[Characteristic polynomial of the augmented Ferrers lattice]
\label{cor:ferrers-characteristic}
Let \(\widehat F_{\mathbf b}=F_{\mathbf b}\sqcup\{\bottom\}\), with the shifted
rank from \cref{lem:minimum-shift}, and let
\[
  X_{\widehat F_{\mathbf b}}(y)
  =\sum_{x\in\widehat F_{\mathbf b}}\mu(\bottom,x)
    y^{b_1-\widehat\rho(x)}
\]
be its characteristic polynomial, where \(\mu\) is the Möbius function
of \(\widehat F_{\mathbf b}\).  Then
\begin{equation}
\label{eq:ferrers-characteristic}
  X_{\widehat F_{\mathbf b}}(y)
  =y^{b_1}-r y^{b_1-1}+(r-1)y^{b_1-2},
\end{equation}
where the last term is understood to be absent when \(r=1\).
Thus the characteristic polynomial is independent of all lower-boundary
entries \(b_2,\ldots,b_r\).  In particular, for the balanced boundary
\(b_i=2r-i\) and \(r\ge2\),
\[
  X_{\widehat C_r}(y)
  =y^{2r-3}(y-1)(y-(r-1)),
\]
where \(\widehat C_r\) denotes the minimum-augmented type-\(B_r/C_r\)
root-poset lattice.
\end{corollary}

\begin{proof}
For \(r=1\), the augmented poset is a chain, with characteristic polynomial
\(y^{b_1}-y^{b_1-1}\).  Hence assume \(r\geq2\).
The augmented lattice has maximal rank \(b_1\).  By
\cref{thm:r-labelling}, its flag \(h\)-number \(\beta(S)\) counts maximal
chains with descent set \(S\).  The path-word description in
\cref{thm:ferrers-family} shows that \(\beta(\varnothing)=1\) and
\(\beta(\{1\})=r-1\): the latter are exactly the chains that start with a
left move and then make all remaining left moves before the right moves.
No descent set can contain both \(1\) and \(2\), and hence
\(\beta(\{1,\ldots,j\})=0\) for every \(j\ge2\).  For a bounded graded
poset \(Q\) of rank \(n\geq1\), Chapoton's identity
\cite[Lemma~6.1]{chapoton2024}, which underlies his description of the
value at \(q=0\) \cite[Theorem~6.2]{chapoton2024}, can be written as
\[
  y^n X_Q(1/y)
  =(1-y)\sum_{a=0}^{n-1}(-1)^a\beta_Q([a])y^a,
  \qquad [0]=\varnothing.
\]
Indeed, \((-1)^a\beta_Q([a])\) is the sum of
\(\mu(\bottom,x)\) over elements of rank at most \(a\); consecutive
differences give the rank-wise M\"obius sums.  Substituting the initial
flag numbers above therefore gives
\[
  y^{b_1}X_{\widehat F_{\mathbf b}}(1/y)
  =1-r y+(r-1)y^2.
\]
Replacing \(y\) by \(1/y\) and multiplying by \(y^{b_1}\) yields
\cref{eq:ferrers-characteristic}.
\end{proof}

\begin{corollary}[Characteristic polynomials of principal ideals]
\label{cor:ferrers-principal-characteristic}
For a cell \(y=(j,d)\in F_{\mathbf b}\), let
\[
  h_y=d-j+1,
  \qquad
  a_y=\min\{r,d\}-j+1.
\]
Then the principal ideal
\([\bottom,y]\subseteq\widehat F_{\mathbf b}\) has characteristic polynomial
\[
  X_{[\bottom,y]}(z)
  =z^{h_y}-a_y z^{h_y-1}+(a_y-1)z^{h_y-2},
\]
where the final term is absent when \(a_y=1\).  In particular, this
polynomial is independent of the intermediate boundary values met below
\(y\).
\end{corollary}

\begin{proof}
If \(a_y=1\), the interval is a chain and its characteristic polynomial
is \(z^{h_y}-z^{h_y-1}\).  Assume now \(a_y\geq2\).
The interval \([\bottom,y]\) inherits the signed EL-labelling.  Its unique
increasing maximal chain starts at the least-labelled atom \((j,j)\) and
then uses right moves, so its empty descent set occurs once.  The atoms of
the interval are exactly \((i,i)\) with
\(j\leq i\leq\min\{r,d\}\), hence there are \(a_y\) of them.  A maximal chain
has descent set \(\{1\}\) precisely when it starts at one of the other
\(a_y-1\) atoms and makes all its left moves before its right moves.  No
descent set contains both \(1\) and \(2\), so all initial descent sets
\(\{1,\ldots,k\}\) with \(k\geq2\) have zero flag \(h\)-number.  The
identity of \cite[Lemma~6.1]{chapoton2024} used in the proof of
\cref{cor:ferrers-characteristic} now gives
\[
 z^{h_y}X_{[\bottom,y]}(1/z)
 =1-a_y z+(a_y-1)z^2,
\]
which is equivalent to the displayed formula.
\end{proof}

\subsection{The type-\texorpdfstring{\(A\)}{A} specialization}

\begin{corollary}[Type-\(A\) companion formula]
\label{cor:type-A}
Let \(P_r^A=\Phi^+(A_r)\) be the standard positive-root poset of type
\(A_r\), graded by root height minus one.  Then
\[
  \HH_{P_r^A,\rho}(q,t)
  =\sum_{w\in\{L,R\}^{r-1}}
    q^{\sum_{j\in D(w)}(j-1)}t^{|D(w)|}.
\]
In particular, its \(q\)-Zeta numerator is coefficientwise nonnegative for every
\(r\geq1\), and
for \(1\leq k\leq\lfloor r/2\rfloor\),
\[
 [t^k]\HH_{P_r^A,\rho}(q,t)
 =q^{k(k-1)}\sum_{h=k}^{r-k}
   \qbinom{h-1}{k-1}\qbinom{r-h}{k}.
\]
The coefficients vanish for \(k>\lfloor r/2\rfloor\), while the constant
coefficient is \(1\).  At \(q=1\), the convolution reduces to
\[
  \HH_{P_r^A,\rho}(1,t)
  =\sum_{k=0}^{\lfloor r/2\rfloor}\binom{r}{2k}t^k
  =\frac{(1+\sqrt t)^r+(1-\sqrt t)^r}{2}.
\]
For \(r\geq2\), this polynomial has exactly
\(\lfloor r/2\rfloor\) distinct negative real zeros.
\end{corollary}

\begin{proof}
Write the positive roots of \(A_r\) as
\(\alpha_{i,j}=e_i-e_{j+1}\) with \(1\leq i\leq j\leq r\).  The map
\(\alpha_{i,j}\mapsto(i,j)\) identifies the root order with
\[
  (i,c)\leq(j,d)\quad\Longleftrightarrow\quad i\geq j,\ c\leq d
\]
on \(F_{\bar{\mathbf b}}\), where
\(\bar{\mathbf b}=(r,\ldots,r)\), and it preserves the rank
\(\operatorname{ht}(\alpha_{i,j})-1=j-i\).  The Ferrers theorem therefore
gives the word formula.  For this constant boundary every word of length
\(r-1\) is admissible.  A binary word of length \(r-1\) with \(k\) \(L\)-runs
is specified by the \(2k\) alternating endpoints of those runs among the
\(r\) gaps, so there are \(\binom{r}{2k}\) such words.  Summing over \(k\)
  and applying the binomial theorem yields the final identity.  The displayed
  coefficient convolution is precisely the constant-boundary case of
  \cref{cor:ferrers-envelope}; at \(q=1\), Vandermonde's identity gives
  \(\binom{r}{2k}\).  To locate
  the zeros, put \(x=\sqrt t\).  Since \(x=1\) is not a root of
  \((1+x)^r+(1-x)^r\), this equation is equivalent to
  \[
    \left(\frac{1+x}{1-x}\right)^r=-1.
  \]
  Writing \(\zeta=(1+x)/(1-x)\), we get \(x=(\zeta-1)/(\zeta+1)\), where
  \(\zeta\) runs over the roots of \(\zeta^r=-1\) with \(\zeta\ne-1\).
  Since \(\lvert\zeta\rvert=1\), each such \(x\) is purely imaginary;
  conjugate values of \(\zeta\) give the same negative value of
  \(t=x^2\), and no other coincidences occur.  Hence there are exactly
  \(\lfloor r/2\rfloor\) distinct negative zeros, which is the degree.
\end{proof}

The explicit zeros, their interlacing, and their limiting distribution
are given in \ref{app:classical-type-A}.

\subsection{Lattice structure and interval formulas}

\begin{remark}[Non-distributivity]
\label{rem:ferrers-nondistributive}
For \(r\geq3\), the augmented Ferrers-cell lattice is not distributive.  Indeed,
with \(x=(2,2)\), \(y=(1,1)\), and \(z=(r,r)\), one has
\[
  x\wedge(y\vee z)=x,
  \qquad
  (x\wedge y)\vee(x\wedge z)=\bottom.
\]
Hence positivity here does not reduce to Chapoton's observation that the
numerator of a distributive lattice is nonnegative
\cite[Section~5]{chapoton2024}.
\end{remark}

\begin{corollary}[Positivity for intervals]
\label{cor:ferrers-interval-positivity}
Let \(I=[x,y]\) be any interval of the augmented Ferrers-cell lattice
\(\widehat{F}_{\mathbf b}=F_{\mathbf b}\sqcup\{\bottom\}\), and give \(I\)
the induced rank function
\(\widehat\rho_I(z)=\widehat\rho(z)-\widehat\rho(x)\).  Then
\[
  \HH_{I,\widehat\rho_I}(q,t)\in\NN[q,t].
\]
\end{corollary}

\begin{proof}
If \(x=y\), the numerator is \(1\) directly from the weak-chain
definition.  Otherwise the interval has positive length.
Restrict the signed labelling in the proof of
\cref{thm:ferrers-family} to the interval \(I\).  The defining increasing
chain and lexicographic-first property are intervalwise, so the restriction
is an EL-labelling, hence an \(R\)-labelling, of the bounded graded poset
\(I\).  Applying the flag expansion \cref{thm:flag-expansion} to \(I\) and
evaluating its flag \(h\)-numbers by \cref{thm:r-labelling} writes
\(\HH_{I,\widehat\rho_I}\) as a sum of monomials \(q^{\qsum{S}}t^{|S|}\)
with nonnegative multiplicities.
\end{proof}

\begin{proposition}[Local path-word formula]
\label{prop:ferrers-interval-word}
Let \(x=(i,c)\leq y=(j,d)\) be nonbottom cells of \(F_{\mathbf b}\), and put
\(h=(i-j)+(d-c)\).  Let \(\mathcal W_{\mathbf b}(x,y)\) be the words with
exactly \(i-j\) letters \(L\) and \(d-c\) letters \(R\) such that, for every
prefix with counts \(L_k,R_k\),
\[
  c+R_k\leq b_{i-L_k}.
\]
For such a word set
\[
  D_{x,y}(w)=\{p\in\{1,\ldots,h-1\}:w_p=R,\ w_{p+1}=L\}.
\]
Then, with the induced rank \(\rho_{x,y}(z)=\rho(z)-\rho(x)\),
\[
  \HH_{[x,y],\rho_{x,y}}(q,t)
  =\sum_{w\in\mathcal W_{\mathbf b}(x,y)}
    q^{\sum_{p\in D_{x,y}(w)}p}t^{|D_{x,y}(w)|}.
\]
\end{proposition}

\begin{proof}[Proof of \cref{prop:ferrers-interval-word}]
For \(x=y\), both sides equal \(1\), with the empty word on the right.
For \(x<y\), maximal chains in \([x,y]\) are precisely the legal shuffles of the
\(i-j\) left moves and \(d-c\) right moves.  After a prefix they are at
\((i-L_k,c+R_k)\), so legality is exactly the displayed boundary condition.
The edge labels along consecutive left moves and consecutive right moves are
strictly increasing; a left move followed by a right move is increasing, while
a right move followed by a left move is a descent.  Thus the descent set is
\(D_{x,y}(w)\).  Apply the bounded flag expansion and the
\(R\)-labelling descent theorem to the interval.
\end{proof}

\begin{corollary}[Upper-degree slice of a nonbottom interval]
\label{cor:ferrers-interval-top-slice}
Let \(x=(i,c)\leq y=(j,d)\) be nonbottom cells and put
\(\ell=i-j\) and \(e=d-c\).  For \(0\leq u\leq \ell-1\), define
\[
  B_u=\min\{e,\,b_{i-u}-c\}.
\]
Then
\[
 [t^{\ell}]\HH_{[x,y],\rho_{x,y}}(q,t)
 =q^{\binom{\ell}{2}}
   \sum_{\substack{1\leq A_0<\cdots<A_{\ell-1}\\ A_u\leq B_u}}
   q^{A_0+\cdots+A_{\ell-1}}.
\]
For \(\ell=0\), the sum consists of the single empty sequence and equals
\(1\).  For \(\ell\geq1\), this coefficient is nonzero if and only if
\(B_u\geq u+1\) for every \(0\leq u<\ell\).  Thus the coefficient at
the upper bound \(\ell\) on descent number is a flagged strict-sequence
enumerator; it need not be the highest nonzero coefficient.
\end{corollary}

\begin{proof}
When \(\ell=0\), the interval is a chain and the assertion is immediate.
In the local descent set \(D_{x,y}(w)\), a left move is counted exactly
when it immediately follows a right move.  A word contributing to
\(t^{\ell}\) therefore has each of its \(\ell\) left moves in a separate
run \emph{and} begins with a right move.  It thus has the form
\[
  R^{a_0}L R^{a_1}L\cdots L R^{a_{\ell-1}}L R^{a_{\ell}},
\]
where \(a_0\geq1\), \(a_u\geq1\) for \(1\leq u<\ell\), \(a_{\ell}\geq0\),
and \(a_0+\cdots+a_{\ell}=e\).  Set
\(A_u=a_0+\cdots+a_u\) for \(0\leq u\leq\ell-1\); these are strictly
increasing with \(A_0\geq1\).  The internal
right-run conditions are exactly \(A_u\leq b_{i-u}-c\), while the final
run condition is \(A_u\leq e\).  Hence the admissible run data are precisely
the displayed strict sequences.  The \((u+1)\)-st left letter is at position
\(A_u+u+1\), so the descent preceding it is at position \(A_u+u\), and the
local formula of \cref{prop:ferrers-interval-word} weights each word by
\(q\) to the sum of these positions times \(t^{\ell}\).  Summing the descent
positions gives \(\binom{\ell}{2}+\sum_uA_u\), proving the formula.
Every admissible sequence satisfies \(A_u\geq u+1\); conversely, choosing
\(A_u=u+1\) proves sufficiency of the stated nonvanishing condition.
\end{proof}

\begin{corollary}[Gaussian-binomial determinant for the local upper slice]
\label{cor:ferrers-interval-top-slice-determinant}
With the notation of \cref{cor:ferrers-interval-top-slice}, assume
\(\ell\geq1\) and use the convention \(\qbinom{N}{k}=0\) for
\(k<0\) or \(N<k\).  Then
\[
 [t^{\ell}]\HH_{[x,y],\rho_{x,y}}(q,t)
 =q^{\binom{\ell+1}{2}}
  \det_{1\leq a,b\leq\ell}
   \qbinom{B_{a-1}-a+b}{1-a+b}.
\]
At \(q=1\) this specializes to the flagged-column determinant in ordinary
binomial coefficients.  For a trapezoidal interval, writing
\(\delta=s-i-c\) gives \(B_u=\min\{e,\delta+u\}\); when
\(e\geq\delta+\ell-1\), the slice reduces to
\(q^{\ell^2}\qbinom{\delta+\ell-1}{\ell}\), and hence to
\(\binom{\delta+\ell-1}{\ell}\) at \(q=1\).
\end{corollary}

\begin{proof}
If \(B_0=0\), the sequence set is empty and the first row of the displayed
determinant is zero, so both sides vanish.  Otherwise all flags are positive.
The preceding corollary counts strictly increasing sequences
\(1\leq A_0<\cdots<A_{\ell-1}\) with \(A_u\leq B_u\), weighted by
\(q^{A_0+\cdots+A_{\ell-1}}\).  Subtracting \(1\) from each term costs the
factor \(q^{\ell}\) and leaves the sequences
\(0\leq A'_0<\cdots<A'_{\ell-1}\) with flags \(A'_u\leq B_u-1\).  A standard
flagged-Schur principal specialization
\cite[Theorem~1.3]{wachs1985} counts the latter as
\[
 \det_{1\leq a,b\leq\ell}\qbinom{(B_{a-1}-1)+1-a+b}{1-a+b}
 =\det_{1\leq a,b\leq\ell}\qbinom{B_{a-1}-a+b}{1-a+b}.
\]
Specifically, take the column shape \((1^{\ell})\), row flags
\(B_{a-1}\), which are weakly increasing in \(a\) because \(\mathbf b\) is
nonincreasing, and variables \(x_j=q^{j-1}\).  The complete homogeneous
polynomial \(h_k(1,q,\ldots,q^{B-1})=\qbinom{B+k-1}{k}\)
gives exactly the displayed entries.
Combining \(q^{\ell}\) with the descent-position offset
\(q^{\binom{\ell}{2}}\) from \cref{cor:ferrers-interval-top-slice} gives the
outer factor \(q^{\binom{\ell+1}{2}}\).
For a trapezoidal boundary, \(\delta=0\) forces \(B_0=0\), and both the
slice and its asserted Gaussian specialization vanish.  For \(\delta\geq1\),
\(B_u=\min\{e,\delta+u\}\) as above; if
\(e\geq\delta+\ell-1\), the shifted sequences are exactly those with
\(A'_u\leq\delta+u-1\), and the substitution \(C_u=A'_u-u\) turns them into
partitions inside an \(\ell\times(\delta-1)\) box, with generating function
\(q^{\binom{\ell}{2}}\qbinom{\delta+\ell-1}{\ell}\).  The total
\(q\)-exponent is then
\(\binom{\ell+1}{2}+\binom{\ell}{2}=\ell^2\).
\end{proof}

\subsection{Boundary comparison and coefficient slices}

\begin{corollary}[The specialization \(q=0\)]
\label{cor:ferrers-qzero}
For every boundary \(\mathbf b\) as above,
\[
  \HH_{F_{\mathbf b},\rho}(0,t)=1+(r-1)t.
\]
For every interval \([x,y]\) with nonbottom lower endpoint,
\(\HH_{[x,y],\rho_{x,y}}(0,t)=1\).
\end{corollary}

\begin{proof}
In the global word formula, a monomial survives at \(q=0\) only when its
descent set is empty or is \(\{1\}\).  The unique word with empty descent set
is \(R^{b_1-1}\).  For each \(1\leq a\leq r-1\), the word
\(L^aR^{b_1-1-a}\) is admissible and has descent set \(\{1\}\), and these
are all possibilities.  This gives the first identity.  In the local interval
formula the exponent is the sum of positive internal descent positions, so
only the empty descent set remains.  Its unique word is the
all-left-then-right shuffle \(L^{i-j}R^{d-c}\), which is legal because the
boundary is
nonincreasing.  Hence the local numerator is \(1\).
\end{proof}

\begin{corollary}[Shape monotonicity]
\label{cor:ferrers-shape-monotonicity}
Fix $r$ and $b_1$.  Let $\mathbf b=(b_1,\ldots,b_r)$ and
$\mathbf b'=(b_1',\ldots,b_r')$ be two nonincreasing boundary sequences
with $b_1'=b_1$ and $b_i\leq b_i'$ for every $i$.  Then
\[
  \HH_{F_{\mathbf b},\rho}(q,t)
  \preceq
  \HH_{F_{\mathbf b'},\rho}(q,t),
\]
where $A\preceq B$ means that $B-A\in\NN[q,t]$.  More generally, if
$x=(i,c)\leq y=(j,d)$ are cells common to both shapes, then
\[
  \HH_{[x,y]_{\mathbf b},\rho_{x,y}}(q,t)
  \preceq
  \HH_{[x,y]_{\mathbf b'},\rho_{x,y}}(q,t).
\]
In particular, for fixed $r$ and $b_1\geq r$, the boundary sequences
$(b_1,r,\ldots,r)$ and $(b_1,\ldots,b_1)$ give coefficientwise lower and
upper bounds for every boundary with that top width.
\end{corollary}

\begin{proof}
The two global word sets have the same word length $b_1-1$.  Every
$\mathbf b$-admissible word satisfies the defining prefix inequalities for
$\mathbf b'$, so $\mathcal W_{\mathbf b}\subseteq\mathcal W_{\mathbf b'}$.
The weights attached to a word are identical, and the first inequality
follows term by term from \cref{thm:ferrers-family}.  The local word sets
have the same prescribed endpoints; the same inclusion of prefix-legal
words and the local formula of \cref{prop:ferrers-interval-word} give the
second inequality.
\end{proof}

\begin{corollary}[A Gaussian-binomial upper bound]
\label{cor:ferrers-envelope}
Put \(M=b_1-1\).  For \(1\leq k\leq r-1\), one has
\[
 [t^k]\HH_{F_{\mathbf b},\rho}(q,t)
 \preceq
 q^{k(k-1)}
 \sum_{h=k}^{\min\{r-1,M-k+1\}}
 \qbinom{h-1}{k-1}\qbinom{M-h+1}{k}.
\]
The right-hand side is the exact \(t^k\)-coefficient for the constant
boundary \((b_1,\ldots,b_1)\).  Thus every Ferrers shape with fixed top
width has an explicit coefficientwise \(q\)-binomial upper bound.
\end{corollary}

\begin{proof}
By \cref{cor:ferrers-shape-monotonicity}, it suffices to evaluate the
constant boundary \(\bar{\mathbf b}=(b_1,\ldots,b_1)\).  Its admissibility
inequalities are automatic: after a prefix, the column coordinate is at most
\(1+\#L+\#R\leq M+1=b_1\).  A word with \(h\) letters \(L\) and \(k\)
\(L\)-runs has the decomposition
\[
  R^{a_0}L^{h_1}R^{a_1}\cdots L^{h_k}R^{a_k},
\]
where \(h_i\geq1\), \(a_0,a_k\geq0\), and \(a_i\geq1\) internally.
The sum of the run-start positions minus one is
\[
  k a_0+\sum_{i=1}^{k-1}(k-i)(h_i+a_i).
\]
Removing the mandatory unit from each \(h_i\) and each internal \(a_i\)
contributes \(k(k-1)\).  The remaining left-run lengths, with total
\(h-k\), give \(\qbinom{h-1}{k-1}\); the remaining right gaps, with total
\(M-h-k+1\), give \(\qbinom{M-h+1}{k}\).  Hence summing the run-start
weight over the compositions gives
\[
 q^{k(k-1)}
 \qbinom{h-1}{k-1}\qbinom{M-h+1}{k}.
\]
Here \(h\geq k\), \(h\leq r-1\), and \(M-h\geq k-1\), which are exactly the
displayed summation bounds.  Summing over \(h\) proves the claim, with an
empty sum interpreted as zero.
\end{proof}

\begin{corollary}[The first $t$-slice for an arbitrary boundary]
\label{cor:ferrers-first-slice}
For the boundary sequence $\mathbf b$ above, write
\([m]_q=1+q+\cdots+q^{m-1}\) for $m\geq1$, and set \([0]_q=0\).  Then
\[
  [t]\HH_{F_{\mathbf b},\rho}(q,t)
  =\sum_{i=2}^{r}[\,b_i-i+1\,]_q.
\]
In particular, for the trapezoidal boundary $b_i=s-i$ this becomes
\[
  [t]\HH_{T_{r,s},\rho}(q,t)
  =\sum_{i=2}^{r}[\,s-2i+1\,]_q.
\]
\end{corollary}

\begin{proof}
A word with exactly one left run has the form
\(w=R^aL^hR^{b_1-1-a-h}\), where $1\leq h\leq r-1$ and
$a\geq0$.  Its initial cell is $(h+1,h+1)$.  The prefix inequalities before
the left run are strongest at its end and read
\(h+1+a\leq b_{h+1}\); after the left run, monotonicity of the boundary
makes the inequalities automatic until the final row, where
\(c\leq b_1\) is automatic.  Thus, for $i=h+1$, the admissible values are
\(0\leq a\leq b_i-i\).  The unique run start is at position $a+1$, so its
weight is $q^a$.  Summing over $i=2,\ldots,r$ gives the first identity; the
trapezoidal specialization follows from $b_i=s-i$.
\end{proof}

\begin{corollary}[Equality in the upper bound]
\label{cor:ferrers-envelope-rigidity}
Let \(\bar{\mathbf b}=(b_1,\ldots,b_1)\) be the constant boundary with
the same top width as \(\mathbf b\).  Then
\[
 [t]\HH_{F_{\bar{\mathbf b}},\rho}(q,t)-[t]\HH_{F_{\mathbf b},\rho}(q,t)
 =\sum_{i=2}^{r}q^{b_i-i+1}[\,b_1-b_i\,]_q.
\]
Consequently, equality in the coefficientwise upper bound of
\cref{cor:ferrers-envelope}
for the whole numerator can occur only for the constant boundary; in fact,
equality of the first \(t\)-slice already forces \(b_i=b_1\) for every
\(i\geq2\).
\end{corollary}

\begin{proof}
The first-slice formula gives the difference as
\(\sum_{i=2}^{r}([b_1-i+1]_q-[b_i-i+1]_q)\).  The elementary identity
\([u]_q-[v]_q=q^v[u-v]_q\) for \(u\geq v\) yields the displayed sum.
Each summand is zero exactly when \(b_i=b_1\), so equality of the first
slice is equivalent to constancy of the boundary.  The final assertion follows from
\cref{cor:ferrers-envelope}.
\end{proof}

\begin{corollary}[The upper-degree $t$-slice for an arbitrary boundary]
\label{cor:ferrers-top-slice}
Put $m=r-1$ and, for $0\leq j\leq m-1$, set
\(B_j=b_{r-j}-r\).  Then
\[
 [t^{r-1}]\HH_{F_{\mathbf b},\rho}(q,t)
 =q^{\binom{m}{2}}
 \sum_{\substack{0\leq A_0<\cdots<A_{m-1}\\ A_j\leq B_j}}
 q^{A_0+\cdots+A_{m-1}}.
\]
For $r=1$ the displayed sum consists of the single empty sequence and
equals $1$.
For \(r\geq2\), the displayed coefficient is nonzero if and only if
\(B_j\geq j\) for all \(0\leq j<m\); otherwise the degree in \(t\)
is strictly less than \(r-1\).
For the trapezoidal boundary $b_i=s-i$, with
\(\delta=s-2r\geq0\), this specializes to
\[
 [t^{r-1}]\HH_{T_{r,s},\rho}(q,t)
 =q^{(r-1)(r-2)}\qbinom{s-r-1}{r-1}.
\]
\end{corollary}

\begin{proof}
The case \(r=1\) is immediate, so assume \(m=r-1\geq1\).
Having $r-1$ left runs forces exactly $r-1$ left letters, each in its own
run.  Write a word as
\(R^{a_0}LR^{a_1}L\cdots LR^{a_{m-1}}LR^{a_m}\), with
$a_0,a_m\geq0$ and $a_j\geq1$ for $1\leq j<m$.  Set
\(A_j=a_0+\cdots+a_j\).  The prefix condition at the end of the $j$-th
right run is precisely $A_j\leq b_{r-j}-r=B_j$, and the intermediate left
moves are then legal by monotonicity.  The $A_j$ are strictly increasing, and
the final right run is determined and nonnegative by the endpoint width.
The run starts contribute
\(\sum_{j=0}^{m-1}(A_j+j)=\binom{m}{2}+\sum_jA_j\), proving the first
identity.  The least strictly increasing nonnegative sequence is
\(A_j=j\), which proves the nonvanishing criterion.
In the trapezoidal case $B_j=\delta+j$.  Writing
\(C_j=A_j-j\) gives weakly increasing integers
\(0\leq C_0\leq\cdots\leq C_{m-1}\leq\delta\); their generating function is
the Gaussian binomial \(\qbinom{\delta+m}{m}\).  The outer factor contributes
the second copy of $\binom{m}{2}$ in the \(q\)-exponent.
\end{proof}

\begin{example}[A nonlinear boundary]
\label{ex:nonlinear-ferrers}
For \(\mathbf b=(5,4,4)\), the boundary is not linear. The admissible words
are
\[
  RRRR,\ LRRR,\ RLRR,\ RRLR,\ LLRR,\ LRLR,\ RLLR,
\]
and the theorem gives
\[
  \HH_{F_{\mathbf b},\rho}(q,t)
  =1+(2+2q+q^2)t+q^2t^2.
\]
The trapezoid with the same top width is \(T_{3,6}\), with boundary
\((5,4,3)\) and numerator \(1+(2+q+q^2)t+q^2t^2\).  The extra word
\(RLLR\) contributes the additional term \(qt\), in accordance with
\cref{cor:ferrers-shape-monotonicity,cor:ferrers-first-slice}.
\end{example}

\section{Further remarks}
\label{sec:discussion}

\subsection{Relation with earlier work}

To our knowledge, \cref{thm:all-roots} and the explicit formulas of
\cref{thm:main,thm:type-D-path,thm:ferrers-family} are new; in
particular, we are not aware of an earlier proof of Chapoton's type-\(B\)
observation.  The ingredients, however, are classical: the lattice
structure of root posets \cite{panyushev2014,panyushev2020}, Stanley's
labelling of upper-semimodular lattices \cite{stanley2012}, Chapoton's
flag expansion \cite{chapoton2024}, and one-boundary lattice-path
enumeration \cite{krattenthaler1993,chu2017qballot}.

At \(q=1\), the numerator records ordinary weak-chain enumeration.  This
differs from the enumeration of antichains and order ideals that is
central to generalized Catalan combinatorics
\cite{athanasiadis2004,panyushev2006}; the variable \(q\) records the
ranks that are forgotten in this specialization.

Related Gaussian-binomial differences occur in the major-index
enumeration of two-row standard Young tableaux \cite[Theorem~2]{keith2019}.
At \(q=1\), the trapezoidal profile agrees with the classical rectangular
chain count \cite{proctor1983,stembridge1986}.  The comparison after
\cref{cor:trapezoid-qone-profile} shows that the rank-weighted numerators
of trapezoids and rectangles differ, and that Stembridge's weighted
identities use other weights.

The Jacobi representation concerns only the binomial-product profiles at
\(q=1\).  Marsh and Williams suggested connections with Chebyshev and
Legendre polynomials for zeta polynomials of nesting nonpartitions
\cite[Section~5]{marshwilliams2022}.  Chen, Yang, and Zhang proved
real-rootedness and interlacing for a different two-parameter family of
generalized Narayana polynomials \cite{chenyangzhang2016}.  Athanasiadis
and Chapoton related the ordinary zeta polynomials of certain posets of
lattice points to Ehrhart polynomials of preorder polytopes
\cite{chapotonathanasiadis2026}; we do not know whether this duality has
a rank-weighted analogue.

By \cref{cor:ferrers-characteristic}, the specialization \(q=0\) gives
\[
  X_{\widehat F_{\mathbf b}}(y)
  =y^{b_1}-r y^{b_1-1}+(r-1)y^{b_1-2},
\]
with the last term absent for \(r=1\).  Thus the characteristic
polynomial forgets the intermediate boundary rows, which remain visible in
the weighted numerator.  The characteristic quasi-polynomials of
arrangements attached to ideals of root systems \cite{tran2019} are
different objects.

\subsection{Research workflow and finite verification}
\label{sec:computational-checks}

Generative AI assisted the development of candidate proofs and formulas
and their subsequent revision; the declaration identifies the systems,
versions, developers and roles.  Verification combined symbolic
arguments, comparison with the cited primary texts, and the independent
exact weak-chain computations described below.

As an independent exact-arithmetic check, we generated the three families
of positive roots of \(B_r\) in simple-root coordinates, matched them
bijectively to \(C_r\), and verified the coordinate order, the join and
meet formulae, and the lower-semimodular rank inequality for
\(1\leq r\leq10\).  The same check enumerates every interval of
\(\wideL_r\) to verify the EL property, computes the \(q\)-Zeta numerator
directly from the defining weak-chain sums, and matches it against the
ballot-word formula, again for \(1\leq r\leq10\).  For these ranks the
word counts agree with \(\binom{2r-2}{r-1}\), the sharp top
\(t\)-coefficient agrees with \cref{cor:sharp-degree}, and the coefficients
of \(t^k\) at \(q=1\) agree with \(\binom{r-1}{k}^2\).

The separate weak-chain calculation stores the polynomial \(E_m(x)\)
for chains of length \(m\) ending at \(x\), using
\[
  E_1(x)=q^{\rho(x)},\qquad
  E_{m+1}(x)=q^{\rho(x)}\sum_{y\leq x}E_m(y).
\]
It then forms \(A_0=1\) and \(A_m=\sum_xE_m(x)\) for
\(1\leq m\leq H+1\), where \(H\) is the maximum rank, and multiplies
the series by the fixed denominator in
\cref{eq:numerator-definition}.
This calculation does not use a labelled-chain or path formula.
The flag expansion and minimum-adjunction identity bound the numerator
degree by \(H\), so coefficients through that degree determine the
numerator; the coefficient of \(t^{H+1}\) is also checked to be zero.

For the uniform root-system argument, a separate Cartan-matrix
calculation verified the lattice and lower-semimodular conditions for
\(A_1,\ldots,A_6\), \(B_2,\ldots,B_6\), \(C_2,\ldots,C_6\),
\(D_4,\ldots,D_8\), \(E_6,E_7,E_8,F_4,G_2\).
The direct weak-chain numerators were compared with the dual-label
descent formula for \(A_1,A_3,B_3,C_3,D_4,\ldots,D_8\) and all five
exceptional types.  The unique increasing-chain property was checked
on every interval in these cases.
The direct and labelled computations use different recursions on the
same root data; root generation itself was additionally compared with
Euclidean models in types \(D\) and \(F_4\), and with the six explicit
positive roots in type \(G_2\).

A further check for \(D_4,\ldots,D_{10}\) compared root-coordinate
dominance with the doubled-cell order and ranks, verified the EL
property on every interval, and compared the coloured-path formula
with both the labelled-chain formula and the direct weak-chain
numerator.  All these computations used exact integer arithmetic.
They supplement the proofs for arbitrary root systems and ranks;
the uniform theorem rests on the root-diamond and flag arguments,
not on these finite ranges.

The six checker sources, README, and source checksums are publicly archived
in the ancillary files of arXiv:2609.24541v2 \cite{wangli2026checks}.

The supplementary archive \texttt{P2252-\allowbreak verification-code.tar.gz}
contains the same checker sources and the three original output files
\texttt{root-\allowbreak lattices.json}, \texttt{root-\allowbreak numerators.json}, and
\texttt{type-d-\allowbreak paths.json}.  The sources require Python~3.10 or later
and use only the standard library.  The README gives invocation commands
and records the finite ranges covered by each check; the archive includes
file checksums and an MIT licence.

\subsection{Further questions}
\label{sec:conclusion}

Minimum adjunction reduces the positivity of the numerator to the
nonnegativity of flag \(h\)-numbers of a bounded poset.  For root posets,
Stanley's labelling of the dual lattice supplies these numbers uniformly,
while the signed and coloured cell labels describe them explicitly in the
classical types.  We mention some questions suggested by this approach.

The uniform theorem leaves the explicit enumerators of the exceptional
root posets to be understood: can their weighted descent formulas be
described by comparably simple combinatorial models?
More generally, which finite graded posets with a unique maximum have
coefficientwise nonnegative \(q\)-Zeta numerators?  Minimum adjunction
together with nonnegative flag \(h\)-numbers gives a sufficient
condition, but not a characterization; Chapoton raised the same question
\cite[Section~5.2]{chapoton2024}.  He also observed apparent positivity
for the parabolic-support posets of finite Coxeter groups, whose
intervals are upper semimodular but which are again unbounded
\cite[Section~5.2]{chapoton2024}.  Whether minimum adjunction and a dual
labelling apply there is a natural test of the present method.

\appendix
\renewcommand{\thetheorem}{\Alph{section}.\arabic{theorem}}
\aliascntresetthe{proposition}
\aliascntresetthe{lemma}
\aliascntresetthe{corollary}
\aliascntresetthe{definition}
\aliascntresetthe{convention}
\aliascntresetthe{remark}
\aliascntresetthe{example}
\section{Classical specializations in type \texorpdfstring{\(B\)}{B}}
\label{app:classical-type-B}

\begin{corollary}[The type-\(B\) profile at \(q=1\)]
\label{cor:real-rooted}
For every \(n\geq0\),
\begin{equation}
\label{eq:legendre-identity}
  B_n(t)=(1-t)^nP_n\!\left(\frac{1+t}{1-t}\right)
  =\sum_{j=0}^{\lfloor n/2\rfloor}
    \binom{n}{2j}\binom{2j}{j}\,t^j(1+t)^{n-2j},
\end{equation}
where \(P_n\) is the Legendre polynomial.  Consequently:
\begin{enumerate}[label=\textup{(\roman*)}]
\item for \(n\geq1\), the zeros of \(B_n\) are simple and negative, and
  the zeros of \(B_n\) and \(B_{n+1}\) strictly interlace; in particular
  the coefficients of \(B_n\) form a strictly log-concave, hence unimodal,
  sequence;
\item \(B_n\) is \(\gamma\)-positive, with
  \(\gamma_j=\binom{n}{2j}\binom{2j}{j}\);
\item one has
  \[
    \sum_{n\geq0}B_n(t)z^n
    =\frac{1}{\sqrt{1-2(1+t)z+(1-t)^2z^2}},
  \]
  and \((n+1)B_{n+1}=(2n+1)(1+t)B_n-n(1-t)^2B_{n-1}\) for \(n\geq1\),
  with \(B_0=1\) and \(B_1=1+t\).
\end{enumerate}
\end{corollary}

\begin{proof}
The case \(\alpha=\beta=0\) of the explicit formula
\cite[Equation~(4.3.2)]{szego1939} reads
\[
  P_n(x)=\sum_{k=0}^{n}\binom{n}{k}^{2}
    \left(\frac{x-1}{2}\right)^{k}\left(\frac{x+1}{2}\right)^{n-k}.
\]
Substituting \(x=(1+t)/(1-t)\), so that \((x-1)/2=t/(1-t)\) and
\((x+1)/2=1/(1-t)\), gives the first equality in
\cref{eq:legendre-identity}.  For the second, the coefficient of \(t^k\)
on the right-hand side is
\[
 \sum_{j}
 \frac{n!}{j!^2(k-j)!(n-k-j)!}
 =\binom{n}{k}\sum_j\binom{k}{j}\binom{n-k}{j}
 =\binom{n}{k}^2,
\]
by Vandermonde's identity.  This also proves (ii).

The map \(t\mapsto(1+t)/(1-t)\) is an increasing bijection from
\((-\infty,0)\) onto \((-1,1)\).  The Legendre polynomials are orthogonal
with respect to a fixed positive weight on \((-1,1)\), so the zeros of
\(P_n\) are simple, lie in \((-1,1)\), and strictly interlace with those
of \(P_{n+1}\) \cite[Theorems~3.3.1--3.3.2]{szego1939}.  Since \(B_n\) has
degree \(n\), the images of these \(n\) zeros are all the zeros of
\(B_n\), which proves (i); strict log-concavity follows from Newton's
inequalities.  Finally, (iii) follows from the generating function
\(\sum_{n\geq0}P_n(x)u^n=(1-2xu+u^2)^{-1/2}\) and from the recurrence
\((n+1)P_{n+1}(x)=(2n+1)xP_n(x)-nP_{n-1}(x)\), with
\(x=(1+t)/(1-t)\) and \(u=(1-t)z\), after multiplying the recurrence
by \((1-t)^{n+1}\).
\end{proof}

\begin{remark}
\label{rem:typeB-zero-law}
Since \(C_r=T_{r,2r}\), the zeros of \(B_n\) are also covered by the case
\(\kappa=0\) of \cref{cor:trapezoid-fixed-offset-zero-law}: as
\(n\to\infty\), their normalized counting measures converge to the
density \(\bigl(\pi\sqrt{-t}\,(1-t)\bigr)^{-1}\) on \((-\infty,0)\), which
is the image of the arcsine law under \(t=(x-1)/(x+1)\).
\end{remark}

\section{Classical trapezoidal profiles}
\label{app:classical-trapezoids}

The following consequences concern only the specialization \(q=1\).
They follow from binomial identities and the classical Jacobi representation
in \cref{cor:trapezoid-qone-profile}.

\begin{proposition}[A rational kernel for the \(q=1\) profiles]
\label{prop:trapezoid-qone-kernel}
For \(A,B\geq0\), put
\[
 J_{A,B}(t)=\sum_{k\geq0}\binom{A}{k}\binom{B}{k}t^k.
\]
Then, as a formal power series,
\begin{equation}
\label{eq:trapezoid-qone-kernel}
 \sum_{A,B\geq0}J_{A,B}(t)x^Ay^B
 =\frac{1}{(1-x)(1-y)-txy}.
\end{equation}
Consequently, for \(A,B\geq1\),
\[
 J_{A,B}=J_{A-1,B}+J_{A,B-1}+(t-1)J_{A-1,B-1},
\]
with \(J_{A,0}=J_{0,B}=1\).  The trapezoidal profile in
\cref{cor:trapezoid-qone-profile} is \(J_{r-1,s-r-1}(t)\).
\end{proposition}

\begin{proof}
Interchanging the sums coefficientwise gives
\[
\begin{aligned}
 \sum_{A,B\geq0}J_{A,B}(t)x^Ay^B
 &=\sum_{k\geq0}t^k
   \frac{x^k}{(1-x)^{k+1}}
   \frac{y^k}{(1-y)^{k+1}}\\
 &=\frac{1}{(1-x)(1-y)-txy},
\end{aligned}
\]
using \(\sum_{A\geq k}\binom{A}{k}x^A=x^k/(1-x)^{k+1}\).
Extracting the coefficient of \(x^Ay^B\) from
\((1-x)(1-y)-txy\) times the series yields the recurrence.
\end{proof}

\begin{proposition}[Fixed-offset algebraic generating functions]
\label{prop:trapezoid-qone-fixed-offset-generating}
For the polynomials \(J_{A,B}(t)\) in
\cref{prop:trapezoid-qone-kernel}, fix \(\kappa\in\NN\) and put
\[
  \Delta_t(z)=\sqrt{\bigl(1+(1-t)z\bigr)^2-4z}.
\]
Then
\begin{equation}
\label{eq:trapezoid-qone-fixed-offset-generating}
  \sum_{n\geq0}J_{n,n+\kappa}(t)z^n
  =\frac{1}{\Delta_t(z)}
   \left(\frac{1+(1-t)z-\Delta_t(z)}{2z}\right)^{\!\kappa}
\end{equation}
as an identity of formal power series in $z$ with coefficients in
\(\mathbb Q[t]\), choosing the square-root branch \(\Delta_t(0)=1\).
The quotient inside the power has a removable singularity at \(z=0\)
and constant term \(1\).  Moreover, for every real $t\geq0$ and every fixed
\(\kappa\in\NN\),
\begin{equation}
\label{eq:trapezoid-qone-fixed-offset-growth}
  \lim_{n\to\infty}J_{n,n+\kappa}(t)^{1/n}=(1+\sqrt t)^2.
\end{equation}
\end{proposition}

\begin{proof}
Set $a=1+(1-t)z$.  Extracting the coefficient of $x^{-\kappa}$ from the
rational kernel after the substitution $y=z/x$ gives
\[
 z^{\kappa}\sum_{n\geq0}J_{n,n+\kappa}(t)z^n
 =[x^{-\kappa}]\frac{1}{a-x-z/x}.
\]
Expanding formally in $(x+z/x)/a$ yields
\[
 [x^{-\kappa}]\frac{1}{a-x-z/x}
 =z^{\kappa}\sum_{j\geq0}\binom{2j+\kappa}{j}
   \frac{z^j}{a^{2j+\kappa+1}}.
\]
The standard Catalan-series identity
\[
 \sum_{j\geq0}\binom{2j+\kappa}{j}u^j
 =\frac{C(u)^{\kappa}}{\sqrt{1-4u}},
 \qquad
 C(u)=\frac{1-\sqrt{1-4u}}{2u},
\]
follows from the Catalan equation \(C(u)=1+uC(u)^2\) by
coefficient extraction.  With $u=z/a^2$, it gives
\eqref{eq:trapezoid-qone-fixed-offset-generating}.

For the growth statement, coefficientwise monotonicity and
\[
  1\leq
  \frac{\binom{n+\kappa}{k}}{\binom{n}{k}}
  \leq\binom{n+\kappa}{\kappa}
  \qquad (0\leq k\leq n)
\]
give
\[
  J_{n,n}(t)\leq J_{n,n+\kappa}(t)
  \leq\binom{n+\kappa}{\kappa}J_{n,n}(t).
\]
Finally, if $c_k=\binom{n}{k}t^{k/2}$, then Cauchy--Schwarz gives
\[
  \frac{(1+\sqrt t)^{2n}}{n+1}
  \leq J_{n,n}(t)\leq(1+\sqrt t)^{2n}.
\]
Taking $n$th roots proves \eqref{eq:trapezoid-qone-fixed-offset-growth}.
\end{proof}

\begin{corollary}[Interlacing at fixed offset]
\label{cor:trapezoid-fixed-offset-interlacing}
Fix \(\kappa\in\NN\), and for \(n\geq1\) write
\(J_{n,n+\kappa}(t)=\prod_{j=1}^{n}(t-\zeta_{n,\kappa,j})\) up to its
positive leading factor, with
\(\zeta_{n,\kappa,1}<\cdots<\zeta_{n,\kappa,n}<0\).  Then the zeros of
\(J_{n,n+\kappa}\) and \(J_{n+1,n+1+\kappa}\) strictly interlace on
\(( -\infty,0)\):
\[
  \zeta_{n+1,\kappa,1}<\zeta_{n,\kappa,1}
  <\zeta_{n+1,\kappa,2}<\cdots<
  \zeta_{n,\kappa,n}<\zeta_{n+1,\kappa,n+1}.
\]
\end{corollary}

\begin{proof}
The Jacobi representation in \cref{cor:trapezoid-qone-profile} gives
\[
  J_{n,n+\kappa}(t)=(1-t)^nP_n^{(0,\kappa)}
      \!\left(\frac{1+t}{1-t}\right).
\]
For fixed nonnegative parameters, consecutive Jacobi polynomials have
simple, strictly interlacing zeros in \((-1,1)\)
\cite[Theorems~3.3.1--3.3.2]{szego1939}.
The change of variables \(t=(x-1)/(x+1)\) is strictly increasing from
\((-1,1)\) onto \((-\infty,0)\), so it preserves the interlacing order.
\end{proof}

\begin{corollary}[Limiting zero distribution at fixed offset]
\label{cor:trapezoid-fixed-offset-zero-law}
Fix \(\kappa\in\NN\), and for \(r\geq2\) set \(s=2r+\kappa\).  If
\(\tau_{r,\kappa,1}<\cdots<\tau_{r,\kappa,r-1}<0\) are the zeros of
\(\HH_{T_{r,2r+\kappa},\rho}(1,t)\), then for every continuous compactly
supported \(f\) on \((-\infty,0)\),
\[
\lim_{r\to\infty}\frac1{r-1}\sum_{j=1}^{r-1}
 f(\tau_{r,\kappa,j})
=\frac1\pi\int_{-\infty}^{0}
  \frac{f(t)}{\sqrt{-t}\,(1-t)}\,dt.
\]
The limit is independent of \(\kappa\).  Equivalently, if
\(j_r/(r-1)\to u\in(0,1)\), then
\[
\tau_{r,\kappa,j_r}\longrightarrow
-\cot^2\!\left(\frac{\pi u}{2}\right).
\]
\end{corollary}

\begin{proof}
For \(n=r-1\), the preceding corollary gives
\[
\HH_{T_{r,2r+\kappa},\rho}(1,t)
=(1-t)^nP_n^{(0,\kappa)}
  \!\left(\frac{1+t}{1-t}\right).
\]
The normalized zero measures of \(P_n^{(0,\kappa)}\) converge to the
arcsine measure on \((-1,1)\) for every fixed \(\kappa\).
Indeed, the Darboux asymptotic \cite[Theorem~8.21.8]{szego1939} is
uniform on compact subintervals of \(0<\theta<\pi\).  Up to a positive
amplitude, its oscillating factor is
\(\cos((n+(\kappa+1)/2)\theta-\pi/4)+O(n^{-1})\).
For a compact interval \([a,b]\subset(0,\pi)\), sampling at successive
integer multiples of \(\pi\) in this phase gives alternating signs for
large \(n\), and hence at least \(n(b-a)/\pi-O(1)\) zeros.
Applying the same bound on compact intervals outside \([a,b]\), and
using the total of \(n\) zeros, gives the matching upper bound after
normalization and passage to the limit.  Letting the compact intervals
increase to \((0,\pi)\) also excludes limiting mass at the endpoints.
This proves uniform zero-angle distribution and thus the arcsine law
under \(x=\cos\theta\).  Pushing it
forward under \(t=(x-1)/(x+1)\) gives the displayed density.  Its distribution
function is \(F(t)=\frac2\pi\arctan(1/\sqrt{-t})\), so the quantile statement
follows.
\end{proof}

\section{Classical zero formulas in type \texorpdfstring{\(A\)}{A}}
\label{app:classical-type-A}

The closed form in \cref{cor:type-A} makes the zeros explicit. The following elementary
consequences show in particular that types \(A\) and \(B\) have the same
limiting zero distribution at \(q=1\).

\begin{corollary}[Zeros in type \(A\)]
\label{cor:type-A-spectrum}
Write
\[
  A_r(t)=\HH_{P_r^A,\rho}(1,t)
  =\sum_{k=0}^{\lfloor r/2\rfloor}\binom{r}{2k}t^k
  \qquad (r\geq1),
\]
and set \(A_0(t)=1\).  Then
\[
  \sum_{r\geq0}A_r(t)z^r
  =\frac{1-z}{1-2z+(1-t)z^2},
\]
so that \(A_1(t)=1\), \(A_2(t)=1+t\), and
\[
  A_r(t)=2A_{r-1}(t)-(1-t)A_{r-2}(t)
  \qquad (r\geq2).
\]
Moreover, the negative zeros of \(A_r\) and \(A_{r+1}\) strictly
interlace on \((-\infty,0)\), in the usual sense for polynomials whose
degrees differ by at most one.
If \(\xi_{r,1},\ldots,\xi_{r,d_r}\) are the zeros, where
\(d_r=\lfloor r/2\rfloor\), then their normalized zero-counting measures
have the same limit as in the type-\(B/C\) case:
for every \(f\in C_c((-\infty,0))\),
\[
  \lim_{r\to\infty}\frac1{d_r}\sum_{j=1}^{d_r}f(\xi_{r,j})
  =
  \frac1\pi\int_{-\infty}^{0}
      \frac{f(t)}{\sqrt{-t}\,(1-t)}\,dt.
\]
When the zeros are ordered increasingly and \(j_r/d_r\to u\in(0,1)\),
\[
  \xi_{r,j_r}\longrightarrow
  -\cot^2\!\left(\frac{\pi u}{2}\right).
\]
\end{corollary}

\begin{proof}
The closed form in \cref{cor:type-A} gives
\[
  A_r(t)=\frac{(1+\sqrt t)^r+(1-\sqrt t)^r}{2}.
\]
Summing the two geometric series proves the displayed generating function
and its recurrence.  For the zero statement, put \(d_r=\lfloor r/2\rfloor\).
The finite zeros are
\[
  \eta_{r,j}=-\tan^2\!\left(\frac{(2j-1)\pi}{2r}\right),
  \qquad 1\leq j\leq d_r.
\]
Indeed, the equation \(A_r(t)=0\), with \(x=\sqrt t\), is equivalent to
\[
  \left(\frac{1+x}{1-x}\right)^r=-1,
\]
and the nonreal unit roots of the right-hand side give the displayed
purely imaginary \(x\)'s and hence the negative values \(\eta_{r,j}\).
This sequence is decreasing; the increasing ordering used in the statement
is \(\xi_{r,j}=\eta_{r,d_r+1-j}\).
The angles \((2j-1)\pi/r\) for rank \(r\) and rank \(r+1\) alternate.  If
\(r=2m\), then
\[
 \frac{(2j-1)\pi}{2m+1}<\frac{(2j-1)\pi}{2m}
 <\frac{(2j+1)\pi}{2m+1}\quad(1\leq j<m),
\]
and the final rank-\((2m+1)\) angle precedes the final rank-\(2m\) angle.
If \(r=2m+1\), the same inequalities with denominators \(2m+1\) and
\(2m+2\) hold for \(1\leq j\leq m\), and there is one additional
rank-\((2m+2)\) angle at the endpoint.  Since
\(-\tan^2(\theta/2)\) is strictly decreasing on \((0,\pi)\), the negative
zeros alternate as claimed.  Finally, the angles
\(\theta_{r,j}=(2j-1)\pi/(2r)\) form a midpoint Riemann grid on
\((0,\pi/2)\), and \(d_r/r\to1/2\).  Hence
\[
 \frac1{d_r}\sum_{j=1}^{d_r}f(\eta_{r,j})
 \longrightarrow
 \frac2\pi\int_0^{\pi/2} f(-\tan^2\theta)\,d\theta.
\]
The substitution \(t=-\tan^2\theta\) gives the displayed density.  Its
distribution function is \(F(t)=\frac2\pi\arctan(1/\sqrt{-t})\), so
quantile convergence yields the final formula.
\end{proof}

\enlargethispage{0in}
\normalsize
\section*{Funding}
This research did not receive any specific grant from funding agencies
in the public, commercial, or not-for-profit sectors.

\section*{Declaration of competing interest}
Qihang Wang reports that Gewu Intelligence Lab provided computing
resources for this work.  Weiye Li declares no known competing financial
interests or personal relationships that could have appeared to influence
the work reported in this paper.

\section*{Acknowledgements}
We acknowledge Gewu Intelligence Lab for providing the
collaborative research environment in which this project was developed.
The laboratory also provided cloud computing resources for the exact
computational checks.

\normalsize
\section*{Declaration of generative AI and AI-assisted technologies in the manuscript preparation process}
During this work, OpenAI GPT-5.6 Sol, Anthropic Claude Fable 5,
xAI Grok 4.6, OpenAI GPT-6 Astra, OpenAI GPT-6.1 Sol,
Anthropic Claude Fable 5.1, and Anthropic Claude Opus 5.5 were used
collectively for mathematical development, candidate proofs and formulas,
proof and reference checking, drafting, and revision. Their contributions
varied across stages of the work.
Both authors have reviewed the manuscript and take full responsibility for its content.

\begingroup
\small
\raggedright
\bibliographystyle{amsplain-fullnames}
\bibliography{references}
\endgroup

\end{document}